\documentclass[12pt]{amsart}
\usepackage{amssymb,amsmath,amscd,graphicx,amsthm,enumerate,latexsym, subfig}
\pdfoutput=1
\usepackage[all]{xy} 
\usepackage[colorlinks]{hyperref}

\newtheorem{theorem}{Theorem}[section]
\newtheorem{corollary}[theorem]{Corollary}
\newtheorem{lemma}[theorem]{Lemma}
\newtheorem{proposition}[theorem]{Proposition}

\theoremstyle{definition}
\newtheorem{definition}[theorem]{Definition}
\newtheorem{remark}[theorem]{Remark}
\newtheorem{claim}{Claim}

\numberwithin{equation}{section}

\theoremstyle{remark}
\newtheorem{example}[theorem]{Example}

\newcommand{\C}{\mathbb{\Bbbk}}

\newcommand{\Rep}{\operatorname{Rep}}
\newcommand{\N}{\mathbb{N}}

\newcommand{\Ext}{\operatorname{Ext}}

\newcommand{\GL}{\textrm{GL}}

\newcommand{\arr}[1]{\ar@{~}[#1] }
\newxyColor{pink}{1.0 0.1 1.0}{rgb}{}
\newxyColor{cyan}{0 0.1 1.0}{rgb}{}
\newxyColor{orange}{1 .7 .2}{rgb}{}

\newcommand{\rank}{\textrm{rank}_k}
\newcommand{\im}{\operatorname{im}}
\newcommand{\Hom}{\operatorname{Hom}}

\newcommand{\X}{\mathfrak{X}}
\newcommand{\EXT}{\ensuremath{\operatorname{EXT}}}
\newcommand{\LCOR}{\operatorname{S^2}}
\newcommand{\RCOR}{\operatorname{T^2}}
\newcommand{\LEND}{\operatorname{S^1}}
\newcommand{\REND}{\operatorname{T^1}}
\newcommand{\ISO}{\operatorname{ISO}}
\renewcommand{\EXT}{\ensuremath{\mathbb{EXT}}}

\title[Generic Modules for Gentle String Algebras]{Generic Modules for Gentle String Algebras}
\author{Andrew T. Carroll}
\address{Department of Mathematics,
		Northeastern University, 
		Boston, MA 02115}
\email{ carroll.a@husky.neu.edu}
\thanks{The author was partially supported by NSF grant number DMS-0801220.}
\date{}
\begin{document}
%%%%%%%%%%%ABSTRACT%%%%%%%%%%%%
\begin{abstract}
We describe the generic modules in each component of the space of $\beta$-dimensional representations of certain string algebras.  In so doing, we calculate dimensions of higher self-extension groups $\Ext^i(M, M)$ for generic modules $M$.  This algorithm lends itself for use in determining tilting modules over gentle string algebras.  
\end{abstract}
%%%%%%%%%%%%%%%%%%%%%%%%%%%%%
\maketitle

\tableofcontents
%%%%%%%%%%%%INTRODUCTION%%%%%%%%%%%%%%%
\section*{Introduction}
Let $\Rep_Q(\beta)$ be the variety of $\beta$-dimensional representations of $Q$.  Kac posed the following question in \cite{Kac}: does there exist an open set $U\subset \Rep_Q(\beta)$ and a decomposition $\beta=\beta(1)+\dotsc+\beta(s)$ such that for each representation $V\in U$, $V=V(1)\oplus \dotsc \oplus V(s)$ with $V(i)$ indecomposable of dimension $\beta(i)$ for each $i$?  Such a decomposition of $\beta$ is referred to as the \emph{canonical decomposition} of $\beta$, and representations in the set $U$ are called \emph{generic}.  Later, Schofield \cite{Schofield} and Derksen-Weyman \cite{DW} gave independent algorithms for determining the canonical decomposition of a given dimension vector for any quiver $Q$, although descriptions of the open set $U$ are, for the most part, still unavailable.  

For quivers with relations, the situation is more intricate.  Representation varieties need not be irreducible, so one needs to consider the canonical decomposition of a dimension vector and describe the generic representations with respect to a given irreducible component.  While it is generally difficult to describe the irreducible components of representation varieties, this problem can be solved for a certain class of zero-relation algebras by relating the representation varieties to varieties of complexes.  

In this article, we give an algorithm to describe the set $U$ in irreducible components of representation varieties for gentle string algebras.  As in \cite{CW}, we determine the irreducible components of these algebras by viewing the representation varieties as products of varieties of complexes, studied by DeConicini and Strickland \cite{DCS}.  This work is a generalization of Kra\'{s}kiewicz and Weyman (\cite{KW}), in which the generic modules for the algebras $A(n)$ are constructed. 

Gentle string algebras have recently seen a resurgence in popularity.  They are an important class of algebras whose representation theory has been very well described (see \cite{AS},\cite{BR}).  More recently, they have appeared in connection with cluster algebras arising from surfaces.  From this point of view, modules without self-extension play an important role (see \cite{ABCP}, \cite{CCS}, \cite{CI}).  

In \cite{Kac}, it is shown that for a quiver $Q$ and dimension vector $\beta$, the decomposition $\beta=\beta(1)+\dotsc+\beta(s)$ is the canonical decomposition of $\beta$ if and only if $\beta(i)$ are Schur roots (i.e., the generic module is indecomposable) and there are no extensions between the generic modules of dimensions $\beta(i)$ and $\beta(j)$ for $i\neq j$.  This result (with some modifications, recalled in section \ref{sec:UDM}) was extended to module varieties of finite dimensional associative algebras by Crawley-Boevey and Schr\"{o}er in \cite{CBS}.  Furthermore (see \cite{G}, \cite{Voigt}), if a $\beta$-dimensional module admits no self extensions, then its $\GL(\beta)$ orbit is open in its irreducible component.  Thus the criterion used to determine the generic modules has interesting connections with tilting theory.  

%The preliminary definitions and notation are introduced in section \ref{sec:preliminaries}.  In section \ref{sec:UDGraph} we construct so-called `Up-and-Down graphs', from which we build `Up-and-Down modules' in section \ref{sec:UDM}.  It is in this section that we explicitly calculate extension groups to show that the Up-and-Down modules are indeed generic.  In section \ref{sec:HIGHEREXT}, we show how to determine the dimensions of the higher extension groups which, when the irreducible component has an open orbit, gives a procedure for determining partial tilting modules.  

The author would like to thank Jerzy Weyman for many helpful insights in the preparation of this article.  In addition, discussions with Ryan Kinser and Kavita Sutar proved very helpful.  
%%%%%%%%%%%%%%%%%%%%%%%%%%%%%%%%
%				PRELIMINARY DEFINITIONS			%
%%%%%%%%%%%%%%%%%%%%%%%%%%%%%%%%

\section{Preliminary Definitions}\label{sec:preliminaries}

Fix an algebraically closed field $\C$.  A quiver $Q=(Q_0, Q_1)$ is a pair consisting of a set of vertices $Q_0$ and a set of arrows $Q_1$.  We denote by $ta$ (resp. $ha$) the tail (resp. head) of the arrow $a$.  A path $p$ in $Q$ is a sequence of arrows $a_s a_{s-1} \dotsc a_1$ such that $ha_i = ta_{i+1}$ for $i=1,\dotsc, s-1$.  We recall the definition of a gentle string algebra (which, in contrast to the original definition, we do not assume to be acyclic).

\begin{definition}
A finite-dimensional $\C$-algebra $A$ is called a \emph{gentle string algebra} if it admits a presentation $\C Q/I$ satisfying the following properties:
\begin{itemize}
\item[i.] each vertex is the head of at most two arrows, and the tail of at most two arrows;
\item[ii.] for each arrow $b\in Q_1$ there is at most one arrow $a$ with $ta=hb$ and at most one arrow $c$ with $hc=tb$ such that $ab\notin I$ and $bc\notin I$;
\item[iii.] for each arrow $b\in Q_1$, there is at most one arrow $a$ with $ta=hb$ and at most one arrow $c$ with $hc=tb$ such that $ab\in I$ and $bc \in I$;
\item[iv.] $I$ is generated by paths of length 2.  
\end{itemize}
\end{definition}

A coloring of a quiver $Q$ is a map $c: Q_1\rightarrow S$ with $S$ a finite set such that $c^{-1}(s)$ is a directed path for each $s\in S$.  The elements of $S$ will be called colors.  For a coloring $c$ of a quiver, define by $I_c$ the ideal $I_c=< ba \mid ha=tb \textrm{ and } c(a)=c(b)>$ (these are monochromatic paths of length two).  We will say that a quiver with relations $\C Q/I$ admits a coloring if $I=I_c$ for some coloring $c$ of $Q$.  Not all zero-relation algebras admit colorings, but the following rather simple result is shown in \cite{CW}.

\begin{proposition}
If $Q$ is acyclic, and $\C Q/I$ is a gentle string algebra then there is a coloring $c$ of $Q$ such that $I=I_c$.
\end{proposition} 

%\begin{example}
%Consider a single three-cycle 
%\[
%	\xymatrix{
%1 \ar[rr]^{a_1} \ar@{<-}[dr]_{a_3} && 2 \ar[dl]^{a_2}\\&3}\]
%There are three distinct (up to renumbering) colorings of the quiver, depicted below with edge type corresponding to color:
%
%\begin{figure}[h!]
%\begin{tabular}{ccc}
%$\xymatrix{
%1 \ar@{~>}[rr]^{a_1} \ar@{<..}[dr]_{a_3} && 2 \ar[dl]^{a_2}\\&3}$ & $\xymatrix{1 \ar@{~>}[rr]^{a_1} \ar@{<~}[dr]_{a_3} && 2 \ar[dl]^{a_2}\\&3}$ & $\xymatrix{1 \ar@{~>}[rr]^{a_1} \ar@{<~}[dr]_{a_3} && 2 \ar@{~>}[dl]^{a_2}\\&3}$\end{tabular}\end{figure}
%The first is not a finite-dimensional string algebra.  Notice that the gentle string algebra $\C Q/<a_3a_2, a_2a_1>$ is not one of the depicted algebras.  
%\end{example}

In this article, we consider gentle string algebras $\C Q/I$ admitting colorings.  

\subsection{Representation Spaces}

Recall that for a dimension vector $\beta \in \N^{Q_0}$ the variety of representations of $\C Q/I$ is given by 
\[
	\Rep_{\C Q/I}(\beta) =\left\{ (V_a)_{a\in Q_1} \in \prod\limits_{a\in Q_1} \Hom_k(k^{\beta_{ta}}, k^{\beta_{ha}}) \mid \sum\limits_{p\in \C Q_1} a_p V(p) = 0 \textrm{ whenever } \sum\limits_{p\in \C Q_1} a_p p \in I\right\}
	\]
where $V(p)$ is the composition of the maps corresponding to the arrows in the path $p$.  The algebraic group $\GL(\beta)=\prod_{i=1}^{n+1} \GL(\beta_i)$ acts linearly on $\Rep_{\C Q/I}(\beta)$, and orbits of this action correspond to isoclasses of modules. 

If there are no relations, then the above variety is simply an affine space.  Otherwise, $\Rep_{\C Q/I}$ need not be irreducible.  In case the algebra does admit a coloring, the irreducible components can be explicitly described by extending results of DeConcini-Strickland \cite{DCS}.  Irreducible components are parametrized by \emph{rank sequences}, defined below.

\begin{definition}
Suppose that $c$ is a coloring of $Q$, and $\beta$ is a dimension vector.  A map $r: Q_1 \rightarrow \N$ is called a \emph{rank map} for $\beta$ (with respect to the coloring $c$) if for each path $a_2a_1$ with $c(a_2)=c(a_1)$, we have $r(a_2)+r(a_1)\leq \beta_{ha_1}.$  A rank map is called \emph{maximal} if it is so under the partial ordering given by $r\leq r'$ if $r(a)\leq r'(a)$ for all $a\in Q_1$.
\end{definition}

Denote by $\Rep_{\C Q/I_c}(\beta, r)$ the set of representations $(V_a)_{a\in Q_1} \in \Rep_{\C Q/I_c}(\beta)$ such that $\rank V_a \leq r(a)$ for $a\in Q_1$.  Notice that if $(V_a)_{a\in Q_1} \in \Rep_{\C Q/I_c}(\beta)$ is a representation of dimension vector $\beta$, then the function $a \mapsto \rank V_a$ is a rank map, and certainly an invariant for the action of $\GL(\beta)$.  In fact, we have the following which is proven in \cite{CW}.
\begin{proposition}[\cite{CW},\cite{DCS}]\label{prop:irrcomponents}
Suppose that $c$ is a coloring of $Q$ such that $\C Q/I_c$ is a gentle string algebra, and let $\beta$ be a dimension vector.  Each irreducible component of $\Rep_{\C Q/I_c}(\beta)$ is of the form $\Rep_{Q, c}(\beta, r)$ where $r$ is a maximal rank map for $\beta$.  
\end{proposition}

By abuse of notation, we will say the pair $(Q, c)$ is a gentle string algebra if $\C Q/I_c$ is.

%%%%%%%%%%%%%%%%%%%%%%%%%%%%%%%
%			UP AND DOWN GRAPH 				       %
%%%%%%%%%%%%%%%%%%%%%%%%%%%%%%%

\section{The Up-and-Down Graph}\label{sec:UDGraph}
In this section, we construct a graph for each irreducible component of $\Rep_{Q, c}(\beta)$ when $(Q, c)$ is a gentle string algebra.  In section \ref{sec:UDM} we will construct a module from each such graph.

Denote by $\X\subset Q_0 \times S$ the set of pairs $(x, s)$ such that there is an arrow $a$ of color $s$ incident to the vertex $x$.  We define a sign function, which will dictate how the graph is constructed.
\begin{definition}
A \emph{sign function} on $(Q,c)$ is a map $\epsilon: \X \rightarrow \{\pm 1\}$ such that if $(x,s_1),(x,s_2)$ are distinct elements in $\X$, then $\epsilon(x,s_1)=-\epsilon(x,s_2)$.  
\end{definition}
The following lemma is not used in the remainder of the article, but is recorded here for completeness.
\begin{lemma}\label{note:signfunction}
If there are no isolated vertices in $Q$, then there are $2^{\lvert Q_0 \rvert}$ sign functions on $(Q, c)$.
\end{lemma}
\begin{proof}
Let $\mathcal{E}$ be the set of all sign functions on $(Q, c)$.  We will define a bijection between this space and $\{\pm 1\}^{Q_0}$.  Namely, for each $x\in Q_0$, select a color $s_x \in \mathcal{C}$ such that $(x, s_x)\in \X$.  If $\epsilon$ is a sign function, denote by $\underline{\epsilon} \in \{\pm 1 \}^{Q_0}$ the vector with $\underline{\epsilon}_x = \epsilon(x,s_x)$.  For $\underline{\epsilon}\in \{\pm 1\}^{Q_0}$, let $\epsilon:\X\rightarrow \{\pm 1\}$ be the extension of the map $\underline{\epsilon}$ by 
\[
\epsilon(x,s)=\begin{cases} \epsilon(x, s_x) & \textrm{ if } s=s_x\\
					-\epsilon(x,s_x) & \textrm{ otherwise} \end{cases}
					\]
These maps are mutual inverses, so indeed $\lvert \mathcal{E} \rvert = \lvert \{\pm 1\}^{Q_0}\rvert =2^{\lvert Q_0 \rvert}$.  
\end{proof}

\begin{definition}\label{def:UDGraph}
Fix a quiver $Q$ with coloring $c$, a dimension vector $\beta$, and a maximal rank map $r$.  For any sign function $\epsilon$ on $(Q, c)$, denote by $\Gamma_{Q,c}(\beta, r, \epsilon)$ the graph with vertices $\{v_i^x\mid x\in Q_0\quad i=1,\dotsc, \beta_x\}$ and edges as follows (see figure \ref{fig:graph} for a visual depiction): for each arrow $a\in Q_1$ and each $i=1,\dotsc, r(a)$
\begin{itemize}
\item[a.] $\xymatrix{ v_i^{ta} \ar@{-}[r] & v_i^{ha}}$ if $\epsilon(ta, c(a)) = 1$, $\epsilon(ha,c(a))=-1$,
\item[b.] $\xymatrix{v_i^{ta} \ar@{-}[r] & v_{\beta_{ha}-i+1}^{ha}}$ if $\epsilon(ta,c(a))=\epsilon(ha,c(a))=1$,
\item[c.] $\xymatrix{v_{\beta_{ta}-i+1}^{ta} \ar@{-}[r] & v_i^{ha}}$ if $\epsilon(ta,c(a))=\epsilon(ha,c(a))=-1$
\item[d.] $\xymatrix{v_{\beta_{ta}-i+1}^{ta}\ar@{-}[r] & v_{\beta_{ha}-i+1}^{ha}}$ if $\epsilon(ta,c(a))=-1$, $\epsilon(ha,c(a))=1$.
\end{itemize}
We will call the graph $\Gamma_{Q,c}(\beta, r, \epsilon)$ an \emph{up-and-down graph}.
\end{definition}
Such a graph comes equipped with a map $w: \operatorname{Edges}(\Gamma_{Q, c}(\beta, r, \epsilon)) \rightarrow Q_1$ where $w(e)=a$ if $e$ is an edge arising from the arrow $a$.  The vertices $v_i^x$ will be referred to as the vertices \emph{concentrated at level $x$}.  Figure \ref{fig:graph} depicts the various edge configurations in $\Gamma_{Q, c}(\beta, r, \epsilon)$ for different choices of $\epsilon$ at the tail and head of an arrow.  

\begin{figure}[h!]
\begin{center}
\includegraphics[width=5in,height=4in]{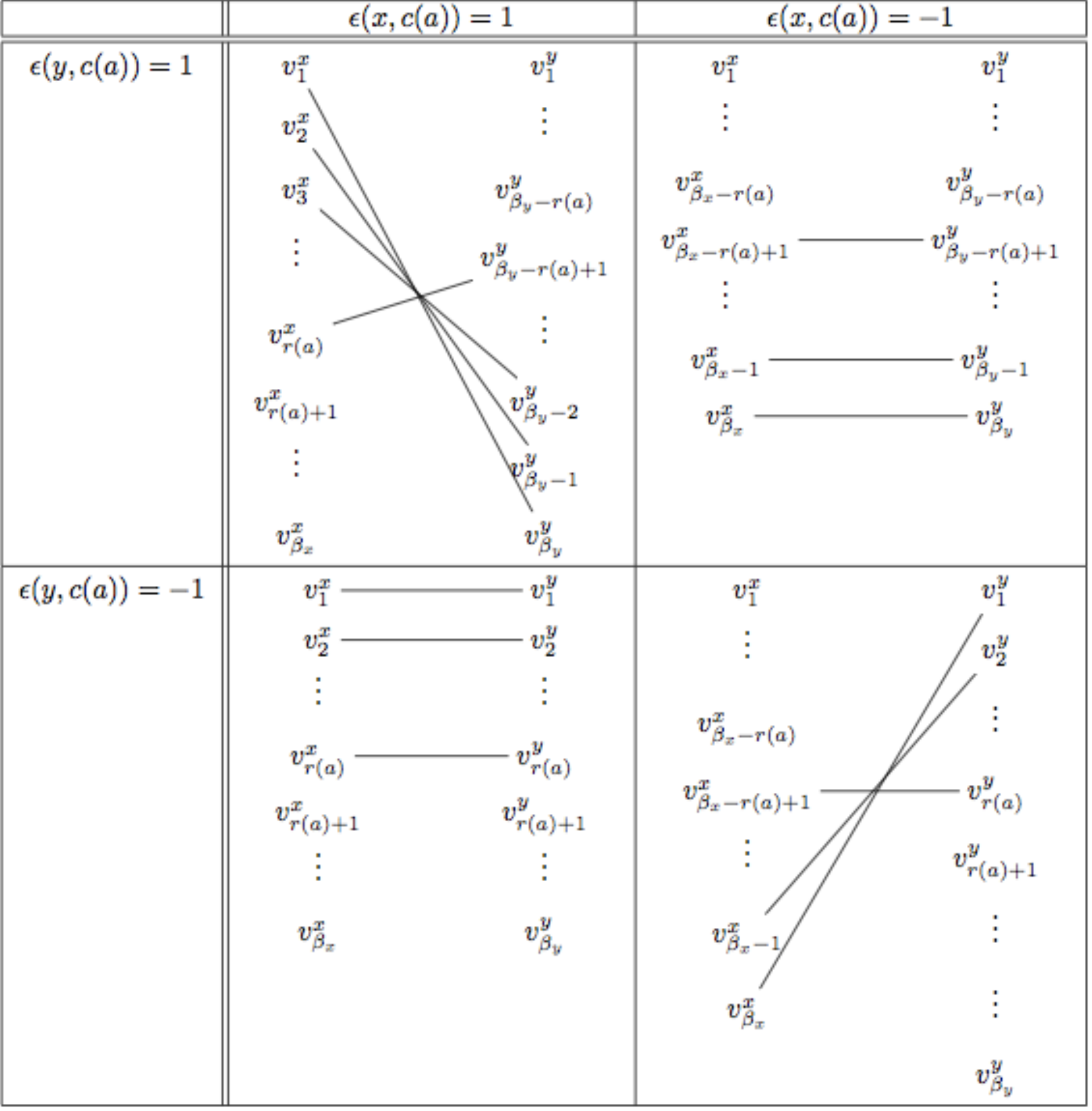}
\end{center}
\caption{A local picture of edges in $\Gamma_{Q, c}(\beta, r, \epsilon)$ with $a\in Q_1$, $x=ta$, $y=ha$, and $s=c(a)$, and varying choices of $\epsilon$.}
\label{fig:graph}
\end{figure}

\begin{proposition}\label{note:UDGraph} Let $\Gamma_{Q, c}(\beta, r, \epsilon)$ be an up-and-down graph.  Then
\begin{itemize}
\item[a.] If a vertex is contained in two edges $e, e'$, then $c(w(e))\neq c(w(e'))$;
\item[b.] Each vertex in $\Gamma_{Q, c}(\beta, r, \epsilon)$ is contained in at most two edges (therefore $\Gamma$ consists of string and band components).
\item[c.] A connected component of an up-and-down graph is again an up-and-down graph.
\end{itemize}
\end{proposition}
\begin{proof}
For part (a), suppose that $v_i^x$ is a vertex in $\Gamma$ incident to two edges $e$ and $e'$ where $c(w(e))=c(w(e'))=s$.  It is clear from the definition of the edges that $w(e)\neq w(e')$.  Since there is at most one outgoing and at most one incoming arrow of color $s$ relative to $x$, it can be assumed that $w(e)=a_1$ and $w(e')=a_2$ where $h(a_1)=t(a_2)=x$ and $c(a_i)=s$.  Suppose that $\epsilon(x,s)=1$ (the other case is identical).  Then by definition \ref{def:UDGraph}, $i\leq r(a_2)$, and $i\geq \beta_x-r(a_1)+1$.  But $r$ is a rank map, so $\beta_x\geq r(a_1)+r(a_2)$.  Therefore, $i\geq r(a_2)+1$ and $i\leq r(a_2)$, a contradiction.  
For part (b), if a vertex $v_i^x$ in $\Gamma$ is contained in three edges, then by part (a) the arrows corresponding to the edges are of three different colors, and all incident to $x$, which is false by assumption that $\C Q/I_c$ is a gentle string algebra.
Finally, suppose that $\gamma$ is a connected component of $\Gamma_{Q, c}(\beta, r,\epsilon)$.  Let us suppose that $\gamma$ has $\beta_x'$ vertices at level $x$ for each $x\in Q_0$, and has $r'(a)$ edges labeled $a$ for each $a\in Q_1$.  Then $\gamma=\Gamma_{Q, c}(\beta', r', \epsilon)$ (this is not simply isomorphism of graphs, but one that preserves the labeling of edges and levels of vertices).  Let us label the vertices in $\Gamma_{Q, c}(\beta',r',\epsilon)$ by $\{w_i^x\mid x\in Q_0, i=1,\dotsc, \beta_x'\}$.  Let $f:\Gamma_{Q, c}(\beta', r', \epsilon)\rightarrow \Gamma_{Q, c}(\beta, r, \epsilon)$ be the homomorphism of graphs defined as follows: $f:w_i^x \mapsto v_{\gamma_i(x)}^x$ where $\gamma_i(x)$ is the $i$-th vertex in $\gamma$ at level $x$.  It is clear that the image of this map is precisely the graph $\gamma$, and that $f$ gives a bijection between $\Gamma_{Q, c}(\beta',r',\epsilon)$ and $\gamma$.
\end{proof}
\begin{remark} It is worth noting that distinct sign functions give rise to a different numbering on the vertices of the graph $\Gamma$, but do not change the graph structure.  In fact, if $\epsilon$ and $\epsilon'$ differ in only one vertex, $x$ (say), the graphs $\Gamma_{Q, c}(\beta, r, \epsilon)$ and $\Gamma_{Q, c}(\beta, r, \epsilon')$ differ only by applying the permutation $i\mapsto \beta_x-i+1$ to the vertices $\{v_i^x \mid i=1,\dotsc, \beta_x\}$.  We will soon see that the families of modules arising from different choices of $\epsilon$ coincide.
\end{remark}

Here we collect some technical definitions and notations to be used concerning these graphs.  We will extend the terminology of Butler and Ringel (\cite{BR}) slightly.  Let $\Gamma_{Q, c}(\beta, r, \epsilon)$ be an up and down graph.  A vertex $v_{j'}^x$ is said to be {\bf above} (resp. {\bf below}) $v_j^x$ if $j>j'$ (resp. $j<j'$).  We will depict the graphs of $\Gamma_{Q, c}(\beta, r, \epsilon)$ in such a way that above and below are literal.  

A vertex $v_j^x$ in $\Gamma_{Q, c}(\beta, r, \epsilon)$ will be referred to as a {\bf source} (resp. {\bf target}) if $t(w(e))=x$ (resp. $h(w(e))=x$) for every edge $e$ containing it.  A {\bf 2-source} (resp. {\bf 2-target}) is a source (resp. target) incident to exactly two edges.  We will denote the sets of such vertices by $S(\Gamma)$, $T(\Gamma)$, $S^2(\Gamma)$, and $T^2(\Gamma)$, respectively.  

To a path $p=v_{i_n}^{x_n} e_n \dotsc v_{i_1}^{x_1} e_1 v_{i_0}^{x_0}$ on $\Gamma_{Q, c}(\beta, r, \epsilon)$, we will associate a sequence $A(p)$ of elements in the set alphabet $Q_1 \cup Q_1^{-1}$ (that is the formal alphabet with characters consisting of the arrows and their inverses), with $$A(p)_i= \begin{cases} w(e_i)  & \textrm{ if } t(w(e_i)) = x_{i-1} \\ w(e_i)^{-1} & \textrm{ if } t(w(e_i))=x_i \end{cases}.$$  Such a path $p$ will be called {\bf direct} (resp. {\bf inverse}) if $A(p)$ is a sequence of elements in $Q_1$ (resp. $Q_1^{-1}$).

Finally, a path $p$ will be called {\bf left positive} (resp. {\bf left negative} if $A(p)_n \in Q_1$ and $\epsilon(x_n, c(e_n))=1$ (resp. $-1$).  Analogously the path is called {\bf right positive} (resp. {\bf right negative}) if $A(p)_1 \in Q_1$ and $\epsilon(x_0, c(e_0))=1$ (resp. $-1$).

\begin{example}\label{example1}
Consider the quiver below with coloring indicated by type of arrow:
\begin{align*}
\xymatrix@C=100pt@R=6ex{
1 \ar@[|(4)][r]^{r_1} \ar@[|(4)]@{..>}[dr]_<<<<<{g_1} & 2 \ar@[|(4)][r]^{r_2} \ar@[|(4)]@{-->}_<<<<<<{p_2}[dr] & 3 \\
4 \ar@[|(4)]@{~>}[r]_{b_1} \ar@[|(4)]@{-->}[ur]^<<<<<{p_1} & 5 \ar@[|(4)]@{..>}[ur]^<<<<<<{g_2} \ar@[|(4)]@{~>}[r]_{b_2} & 6 }
\end{align*}
Let us say that the color of the arrow $a_i$ is $a$ in the above picture.  Let $\beta,r$ be the pair depicted in the following:
\begin{align*}
\xymatrix@C=100pt@R=2ex{
*+[F]{3} \ar@[|(4)][r]|*+[o][F]{3} \ar@[|(4)][ddddr]|<<<<<<<<<<<<*+[o][F]{2} & *+[F]{4} \ar@[|(4)][r]|*+[o][F]{1} \ar@[|(4)][ddddr]|<<<<<<<<<<*+[o][F]{2}& *+[F]{1} \\
&&\\
&&\\
&&\\
*+[F]{2} \ar@[|(4)][r]|*+[o][F]{2} \ar@[|(4)][uuuur]|<<<<<<<<<<<<*+[o][F]{2} & *+[F]{3} \ar@[|(4)][uuuur]|<<<<<<<<*+[o][F]{1} \ar@[|(4)][r]|*+[o][F]{1} & *+[F]{2} }
\end{align*}
and $\epsilon^{-1}(1)=\{(1,g), (2,p), (3,g), (4, b), (5,b), (6,p)\}$ (so $\epsilon^{-1}(-1)$ is the complement in $\X$).  Then $\Gamma_{Q, c}(\beta, r,\epsilon)$ takes the following form:

\xymatrix@R=5pt@C=90pt{
v_1^{(1)} \ar@[|(4)]@{-}[ddrr]^>>>>>>>>>>>{r_1} \ar@[|(4)]@{..}[dddddrr]^<<<<<<<<<<<<<<<<<<<<<<<<<<{g_1}&& v_1^{(2)} \ar@[|(4)]@{--}[ddddddrr]^<<<<<<<<<{p_2} && v_1^{(3)} \\
v_2^{(1)} \ar@[|(4)]@{-}[rr]^>>>>>>>>>{r_1} \ar@[|(4)]@{..}[dddddrr]^<<<<<<<<<<<<<<<<<<<<<{g_1} && v_2^{(2)} \ar@[|(4)]@{--}[ddddrr]^<<<<<<<<<{p_2} && \\
v_3^{(1)} \ar@[|(4)]@{-}[uurr]^>>>>>>>>>>>{r_1} && v_3^{(2)}  && \\
		&&		v_4^{(2)} \ar@[|(4)]@{-}[uuurr]^>>>>>>>>>>>>>>>>{r_2}  && \\
&&&&\\
v_1^{(4)}\ar@[|(4)]@{~}[ddrr]_>>>>>>>>>>>>>>>>>>{b_1} \ar@[|(4)]@{--}[uuurr]^<<<<<<<<<<<<<<<<<{p_1} && v_1^{(5)} \ar@[|(4)]@{~}[rr]^{b_2} && v_1^{(6)} \\
v_2^{(4)} \ar@[|(4)]@{--}[uuurr]^<<<<<<<<<<<<<<<<<<<{p_1} \ar@[|(4)]@{~}[rr]_>>>>>>>>>>>>>>{b_1} && v_2^{(5)}  && v_2^{(6)}\\
		&&		v_3^{(5)} \ar@{..}@[|(4)][uuuuuuurr]_<<<<<<<<<<<<<<<<<<<{g_2}		&& }
		
\begin{itemize}
\item[i.] The vertices $v_1^{(1)}, v_2^{(1)}, v_3^{(1)},v_1^{(4)}, v_2^{(4)}$ are sources, and $v_3^{(2)}, v_2^{(5)}, v_1^{(3)}, v_1^{(6)}, v_2^{(6)}$ are targets.  
\item[ii.] The path $v_2^{(6)}e_2v_1^{(2)}e_1v_3^{(1)}$ with $w(e_1)=r_1$ and $w(e_2)=p_2$ is a direct path that is left positive (since $\epsilon(6,p_2)=1$), and right negative; while $v_1^{(4)}e_2v_3^{(2)}e_1 v_1^{(1)}$ with $w(e_1)=p_1$, $w(e_2)=r_1$ is not a direct path.
\end{itemize}
\end{example}

\subsection{Some Combinatorics for Up-and-Down Graphs}

The proof of the main theorem requires an explicit description of the projective resolution of the modules arising from up-and-down graphs.  In this section, we collect some technical lemmas concerning the structure of the graphs $\Gamma_{Q, c}(\beta, r, \epsilon)$ to be used in describing the projective resolution.  
\begin{lemma}\label{lem:rightleftconnection} Let $v_j^x$ be a vertex in $\Gamma_{Q, c}(\beta, r, \epsilon)$, and suppose that $$p=v_j^x e_l v_{i_{l-1}}^{x_{l-1}} \dotsc v_{i_1}^{x_1} e_1 v_i^y$$ is a left direct path ending in $v_j^x$.
\begin{itemize}
\item[A.] If $p$ is left negative direct, and $v_{j'}^x$ is above $v_j^x$, then there is a left negative direct path $$p'=v_{j'}^x e_{l-1}' v_{i_{l-1}'}^{x_{l-1}} \dotsc v_{i_1'}^{x_1} e_1' v_{i'}^y$$ with $A(p')=A(p)$.  Furthermore,
\begin{itemize}
\item[A1.] $v_{i'}^y$ is above $v_i^y$ if and only if $\epsilon(y,c(w(e_1')))=1$;
\item[A2.] $v_{i'}^y$ is below $v_i^y$ if and only if $\epsilon(y, c(w(e_1')))=-1$.
\end{itemize}
\item[B.] If $p$ is left positive direct, and $v_{j'}^x$ is below $v_j^x$, then there is a left positive direct path $$p'=p'=v_{j'}^x e_{l-1}' v_{i_{l-1}'}^{x_{l-1}} \dotsc v_{i_1'}^{x_1} e_1' v_{i'}^y$$ with $A(p')=A(p)$.  Furthermore, 
\begin{itemize}
\item[B1.] $v_{i'}^y$ is below $v_j^x$ if and only if $\epsilon(y,c(w(e_1')))=-1$;
\item[B2.] $v_{i'}^y$ is above $v_j^x$ if and only if $\epsilon(y,c(w(e_1')))=1$.
\end{itemize}
\end{itemize}
\end{lemma}
\begin{proof}
We will prove this lemma by induction on the length of $p$.  Suppose that $p=v_j^x e_1 v_i^y$ with $A(p) =a$.  If $p$ is left negative direct, then $\epsilon(x,c(a))=-1$.  By definition of the graph $\Gamma$, then, $j\leq r(a)$.  But $v_{j'}^x$ is above $v_j^x$ if and only if $j'<j$.  By definition \ref{def:UDGraph} (a), (c), there is an edge $e_1'$ terminating at $v_{j'}^x$ labeled $a$, so $p'=v_{j'}^x e_1' v_{i'}^y$.  If $\epsilon(y,c(a))=1$, $i=j$ and $i'=j'$, so indeed $i'<i$, implying that $v_{i'}^y$ is above $v_i^y$.  On the other hand, if $\epsilon(x,c(a))=-1$, then $i=\beta_x-j+1$ and $i'=\beta_x-j'+1$, so $i'>i$, and $v_{i'}^y$ is below $v_i^y$.  The other direction is also clear for [A1] and [A2]. 

Now suppose that $p$ is left positive direct of length one, i.e., $\epsilon(x,c(a))=1$.  Write $p=v_i^y e_1 v_{j}^x$.  By definition of $\Gamma$, then, $j\geq\beta_x-r(a)+1$.  Suppose that $j=\beta_x-\hat{j}+1$, and $j'=\beta_x-\hat{j'}+1$.  Since $v_{j'}^x$ is below $v_j^x$, we have that $j'=\beta_x-\hat{j'}+1>\beta_x-\hat{j}+1 = j$, so $\hat{j'}<\hat{j}$.  Indeed, $j'>\beta_x-r(a)+1$, so by definition \ref{def:UDGraph} (b) or (d), there is an edge $e_1'$ labeled $a$ terminating at $v_{j'}^x$.  $\epsilon(y,c(a))=-1$ if and only if $i=\beta_y-\hat{j}+1$ and $i'=\beta_y-\hat{j'}+1$, i.e., $i'>i$, so $v_{i'}^y$ is below $v_i^y$.  $\epsilon(y,c(a))=1$ if and only if $i=\hat{j}$ and $i'=\hat{j'}$, i.e., $i'<i$, so $v_{i'}^y$ is above $v_i^y$.  

Now assume that [A] and [B] are true for all paths of length at most $l-1$.  Suppose that 
$$p=v_j^x e_l v_{j_{l-1}}^{y_{l-1}} e_{l-1} v_{j_{l-2}}^{y_{l-2}} e_{l-2} \dotsc e_1 v_i^y$$ 
is left negative direct, and $v_{j'}^x$ is above $v_j^x$.  By the first step, there is a path $v_{j'}^x e_l' v_{j_{l-1}'}^{y_{l-1}}$ with $w(e_l)=w(e_l')$.  
\begin{itemize}
\item[Case 1:] $\epsilon(y_{l-1}, c(w(e_l')))=1$ if and only if $v_{j_{l-1}'}^{y_{l-1}}$ is above $v_{j_{l-1}}^{y_{l-1}}$.  But by proposition \ref{note:UDGraph}, $\epsilon(y_{l-1}, c(w(e_l)))=-\epsilon(y_{l-2}, c(w(e_{l-1})))$, so $$\tilde{p}= v_{j_{l-1}}^{y_{l-1}} e_{l-1} \dotsc e_1 v_i^y$$ is left negative direct.  So by the inductive hypothesis, since $\tilde{p}$ is of length $l-1$, we have a path $$\tilde{p'} = v_{j_{l-1}'}^{y_{l-1}} e_{l-1}' \dotsc e_1' v_{i'}^y$$ with $A(\tilde{p})=A(\tilde{p'})$.  Taking $p'=e_l' \tilde{p'}$, we have a left negative direct path $p'$ terminating in $v_{j'}^x$.  Again, by the inductive step, $v_{i'}^y$ is above (resp. below) $v_i^y$ if and only if $\epsilon(y, c(w(e_1)))=1$ (resp. $-1$).
\item[Case 2:] $\epsilon(y_{l-1}, c(w(e_l')))=-1$ if and only if $v_{j_{l-1}'}^{y_{l-1}}$ is below $v_{j_{l-1}}^{y_{l-1}}$.  By proposition \ref{note:UDGraph}, $\epsilon(y_{l-1}, c(w(e_l)))=-\epsilon(y_{l-2}, c(w(e_{l-1})))$, so $$\tilde{p} = v_{j_{l-1}}^{y_{l-1}} e_{l-1} \dotsc e_l v_i^y$$ is left positive direct.  By the inductive hypothesis, there is a path $$\tilde{p'} = v_{j_{l-1}'}^{y_{l-1}} e_{l-1}' \dotsc e_1' v_{i'}^y$$ with $A(\tilde{p})=A(\tilde{p'})$.  Taking $p'=v_{j'}^x e_l' \tilde{p'}$, we have a left negative direct path $p'$ terminating in $v_{j'}^x$.  By the hypothesis, $v_{i'}^y$ is above (resp. below) $v_i^y$ if and only if $\epsilon(y, c(w(e_1)))=1$ (resp. $-1$).  
\end{itemize}
The same argument hold if $p$ is left positive direct, interchanging the terms `above' and `below'.   
\end{proof}

Here we collect some properties that determine what types of extremal vertices occur in which levels.
\begin{lemma}\label{note:corners} Let $\Gamma_{Q, c}(\beta, r, \epsilon)$ be an up-and-down graph, and let $a_1, a_2, b_1, b_2\in Q_1$ be colored arrows as indicated in the figure:
\begin{align*}
\xymatrix@R=1pt{
\ar[dr]^{a_1} &  & \\
&y \ar[ur]^{a_2}\ar@{..>}[dr]_{b_2}\\
\ar@{..>}[ur]_{b_1} & & }
\end{align*}
\begin{itemize}
\item[i.] If $v_j^y$ is a 2-source (resp. 2-target), then $r(a_1)+r(b_1)>\beta_y$ (resp. $r(a_2)+r(b_2)>\beta_y$);
\item[ii.] Let $m_1=\max\{r(a_1), r(b_2)\}$ and $m_2=\max\{r(b_1), r(a_2)\}$.  Then if $v_j^y$ is an isolated vertex, $m_1+m_2<\beta_y$ (in particular, there are neither 2-sources not 2-targets vertices at level $y$);
\item[iii.] If $v_i^y$ is a 1-target contained in an edge labeled by $a_1$ (resp. $b_2$), then $r(b_1)+r(b_2)<\beta_x$ (resp. $r(a_1)+r(a_2)<\beta_x$).  
\end{itemize}
\end{lemma}
\begin{proof}
We prove only (iii), since the others are similar.  Suppose that $v_i^y$ is a 1-target contained in an edge labeled $a_1$.  If $r(b_1)+r(b_2)=\beta_y$, then each vertex at level $y$ would be contained in an edge (either labeled $b_1$ or $b_2$), including $v_i^y$.  But this contradicts the assumption.  
\end{proof}

\begin{lemma}\label{lem:maximalrank}
Suppose that there is a sequence of arrows $a_1, a_2, a_3 \in Q_1$ with $c(a_1)=c(a_2)=c(a_3)$, $h(a_1)=t(a_2)=x_1$, and $h(a_2)=t(a_3)=x_2$.  If $r$ is a maximal rank map, then we have the following:
\begin{itemize}
\item[i.] if $r(a_1)+r(a_2)<\beta_{x_1}$ then $r(a_2)+r(a_3)=\beta_{x_2}$
\item[ii.] if $r(a_2)+r(a_3)<\beta_{x_2}$ then $r(a_1)+r(a_2)=\beta_{x_1}$.
\end{itemize}

\end{lemma}
\begin{proof}
Suppose that both $r(a_1)+r(a_2)<\beta_{x_1}$ and $r(a_2)+r(a_3)<\beta_{x_2}$.  Define by $r^+$ the rank map with $r^+(a_2)=r(a_2)+1$ and $r^+(b)=r(b)$ otherwise.  $r^+$ is a rank map and $r^+>r$, contradicting the assumption of maximality of $r$.
\end{proof}

%%%%%%%%%%%%%%%%%%%%%%%%%%%%%%%%
%			UP AND DOWN MODULES				%
%%%%%%%%%%%%%%%%%%%%%%%%%%%%%%%%

\section{Up-and-Down Modules}\label{sec:UDM}
We will now define a module (or family of modules) $M_{Q, c}(\beta, r)$ based on two additional parameters, later proving that the isomorphism class of this module (or family of modules) is independent of these parameters.  Fix $Q, c, \beta, r, \epsilon$ as described above.  Recall that proposition \ref{note:UDGraph} guarantees $\Gamma_{Q, c}(\beta, r, \epsilon)$ is comprised of strings and bands.  Let $B(\Gamma)$ be the set of bands and fix a function $\Theta:B(\Gamma) \rightarrow \operatorname{Vert}(\Gamma_{Q, c}(\beta, r, \epsilon))$ with $\Theta(b)$ a target contained in the band $b$.
\begin{definition}
For $\lambda \in (\C^*)^{B(\Gamma)}$, denote by $M_\lambda:=M_{Q, c}(\beta, r, \epsilon, \Theta)_\lambda$ the representation of $Q$ given by the following data.  The space $(M_\lambda)_x$ is a $\beta_x$-dimensional $\C$ vector space together with a fixed basis $\{e_j^x\}_{j=1,\dotsc, \beta_x}$.  The linear map $(M_\lambda)_a: (M_\lambda)_{ta}\rightarrow (M_\lambda)_{ha}$ is defined as follows: if $v_j^{ta}$ and $v_k^{ha}$ are joined by an edge $e$ labeled $a$, then 
\begin{align*}
(M_\lambda)_a: e_j^{ta} \mapsto \begin{cases} \lambda_be_k^{ha} & \textrm{ if there is a band b with } \Theta(b)=v_k^{ha} \textrm{ and } \epsilon(ha, c(a))=1 \\ e_k^{ha} & \textrm{ otherwise.}\end{cases}
\end{align*}
If there is no such edge, then $(M_\lambda)_a:e_j^{ta}\mapsto 0$.  If there are no bands in $\Gamma_{Q, c}(\beta, r, \epsilon)$, denote by $M_{Q, c}(\beta, r, \epsilon, \Theta)$ the subset of $\Rep_{Q}(\beta)$ containing this module.  If there are bands, then denote by $M_{Q, c}(\beta, r, \epsilon, \Theta)$ the set of all modules $M_{Q, c}(\beta, r, \epsilon, \Theta)_\lambda$ for $\lambda\in (\C^*)^{B(\Gamma)}$.     
\end{definition}
\begin{example}
Continuing with example \ref{example1}, let $b$ be the unique band in $\Gamma_{Q, c}(\beta, r, \epsilon)$, and take $\Theta(b)=v_1^{(6)}$.  For $\lambda \in \C^*$, the module $M_{Q, c}(\beta, r, \epsilon, \Theta)_\lambda$ is given by the following:
\begin{figure}[h!]
\includegraphics[trim=.1in 1in .2in .6in]{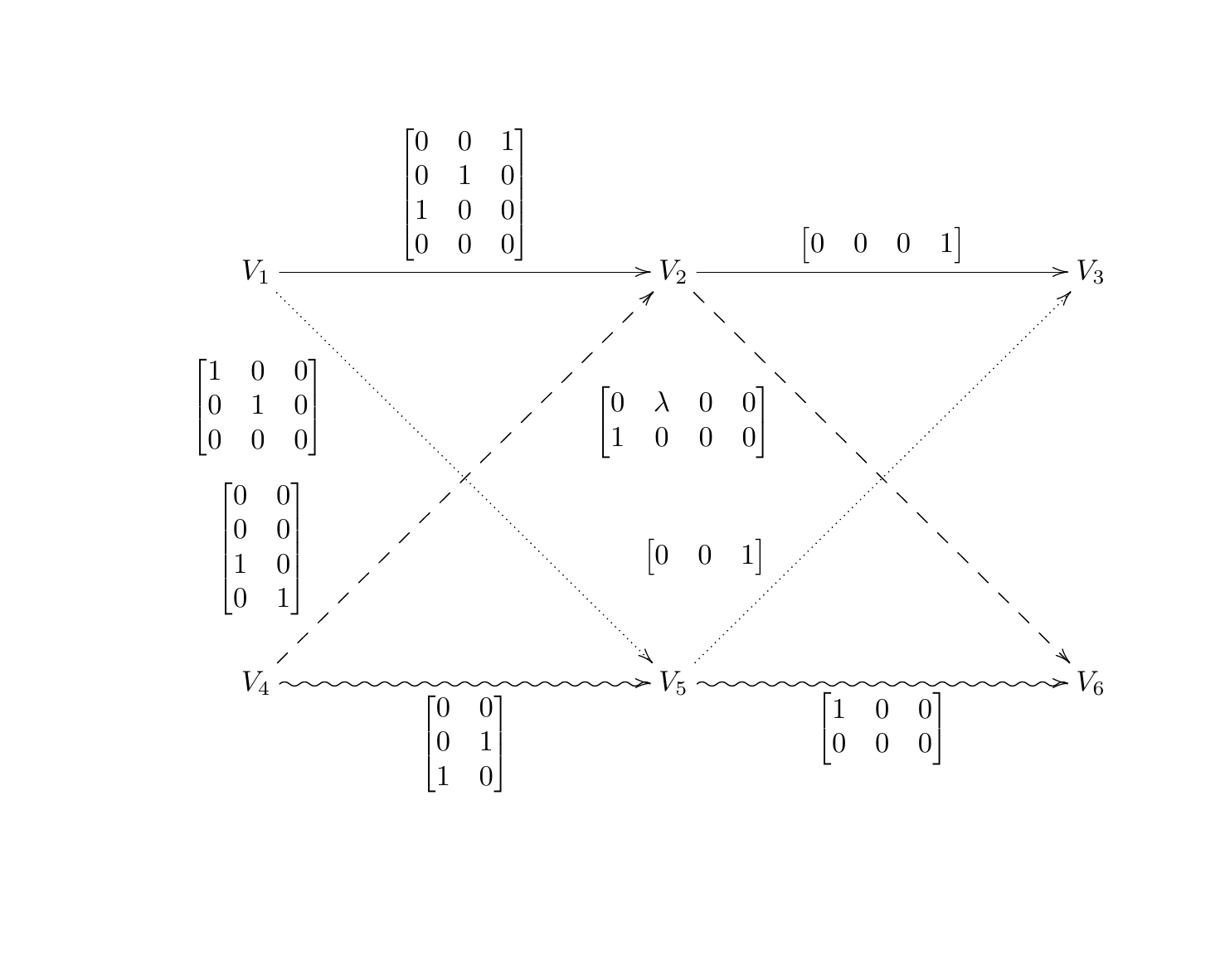}
\end{figure}
\end{example}
\begin{proposition}
Every representation in the set $M_{Q, c}(\beta, r, \epsilon, \Theta)$ is a representation of the gentle string algebra $(Q, c)$.
\end{proposition}
\begin{proof}
If $a_1, a_2\in Q_1$ are arrows with $ha_1=ta_2$ and $c(a_1)=c(a_2)$, then by proposition \ref{note:UDGraph} there is no path $\xymatrix{v_1 \ar@{-}[r]^{e_1} & v_2 \ar@{-}[r]^{e_2} & v_3}$ in $\Gamma$ with $w(e_1)=a_1$ and $w(e_2)=a_2$.  Therefore, $a_2a_1(e_i^x)=0$ for all $x\in Q_0,\ i=1,\dotsc, \beta_x$.  Since $I_c$ is generated by precisely these relations, each module in $M_{Q, c}(\beta, r, \epsilon, \Theta)$ is indeed a $\C Q/I_c$ module.  
\end{proof}

The definition appears highly dependent on both $\epsilon$ and the choice of distinguished vertices $\Theta$.  In the following proposition, we show that the family does not depend on $\Theta$.

\begin{proposition}
The family $M_{Q, c}(\beta, r, \epsilon, \Theta)$ does not depend on the choice of vertices $\Theta$.  
\end{proposition}
\begin{proof}
This proof is a simple consequence of \cite{BR} (cf. Theorem page 161).  Indeed, suppose that $b$ is a band component of some $M_{Q, c}(\beta, r, \epsilon, \Theta)_\lambda$.  Denote by $M_{b,\lambda_b}$ the submodule corresponding to this band, and $\omega_b$ a cyclic word in $Q_1 \cup Q_1^{-1}$ which yields this band.  Recall that Butler and Ringel produce, for each such cyclic word, a functor $F_{\omega_b}$ from the category of pairs $(V, \varphi)$ with $V$ a $k$-vector space and $\varphi: V \rightarrow V$ and automorphism, to the category $\Rep_{Q, c}$.  The indecomposable module $M_{b, \lambda_b}$ is isomorphic to the image under this functor of the pair $(k, \lambda_b)$ where $\lambda_b: x \mapsto \lambda_b \cdot x$.  Butler and Ringel show that the family $M_{b, \lambda_b}$ for $\lambda_b \in \C^*$ is independent of cyclic permutation of the word $\omega_b$.  (They show the image of the functor $F_{\omega_b}$ itself is independent cyclic permutations of $\omega_b$.)  Therefore, any choice of vertices $\Theta$ yields the same family $M_{Q, c}(\beta, r, \epsilon)$.  
\end{proof}
Henceforth, we drop the argument $\Theta$, and when necessary we make a particular choice of said function. 

\subsection{Main Theorem and Consequences}
The following statement constitutes the main theorem of the article, although its corollaries are the applicable results.  In particular, it allows us to show that the representations $M_{Q, c}(\beta, r)$ are generic.  The remainder of the article will be primarily concerned with proving this theorem.
\begin{theorem}\label{thm:vanishingextbands} Let $B(\Gamma)$ be the set of bands for the graph $\Gamma_{Q, c}(\beta, r, \epsilon).$
\begin{itemize}
\item[a.] Suppose that $\lambda, \lambda' \in (\C^*)^{B(\Gamma)}$ with $\lambda_b\neq \lambda'_{b'}$ for all $b,b'\in B(\Gamma)$.  Then $$\dim \Ext^1_{\C Q/I_c}(M_{Q, c}(\beta, r, \epsilon)_\lambda, M_{Q, c}(\beta, r, \epsilon)_{\lambda'})=0.$$  
\item[b.] Suppose that $\Gamma_{Q, c}(\beta, r, \epsilon)$ consists of a single band component.  Then $$\dim \Ext^1_{\C Q/I_c}(M_{Q, c}(\beta, r, \epsilon)_\lambda, M_{Q, c}(\beta, r, \epsilon)_\lambda) = 1.$$
\end{itemize}
\end{theorem}

\begin{corollary}
If $B(\Gamma)=\emptyset$, i.e., $\Gamma_{Q, c}(\beta, r, \epsilon)$ consists only of strings, then the unique element $M\in M_{Q, c}(\beta, r)$ has a Zariski open orbit in $\Rep_{Q, c}(\beta, r)$.
\end{corollary}
\begin{proof}
If there are indeed no band components in $\Gamma_{Q, c}(\beta, r, \epsilon)$, then by theorem \ref{thm:vanishingextbands} part (a), we have $\Ext^1_{\C Q/I_c}(M, M) = 0$.  The corollary then follows by (\cite{G} Corollary 1.2, \cite{Voigt}).  
\end{proof}

 If there are band components, then the analogous corollary is more subtle, although the result is essentially the same.  Namely that the union of the orbits of all elements in $M_{Q,c}(\beta, r)$ is dense in its irreducible component.   The proof relies on some auxiliary results due to Crawley Boevey-Schr\"{o}er \cite{CBS}, and so we exhibit those first.  Let $\C Q/I$ be an arbitrary quiver with relations.   Suppose that $C_i \subset \Rep_{\C Q/I}(\beta(i))$ are $\GL(\beta(i))$-stable subsets for some collection of dimension vectors $\beta(i)$, $i=1,\dotsc, t$, and denote by $\beta = \sum_i \beta(i)$ the sum of the dimension vectors.  Define by $C_1 \oplus \dotsc \oplus C_t$ the $\GL(\beta)$-stable subset of $\Rep_{\C Q/I}(\beta)$ given by the set of all $\GL(\beta)$ orbits of direct sums $M_1 \oplus \dotsc\oplus M_t$ with $M_i \in C_i$.  
 \begin{theorem}[Theorem 1.2 in \cite{CBS}]
For an algebra $\C Q/I$, $C_i \subset \Rep_{\C Q/I}(\beta(i))$ irreducible components and $t$ defined as above, the set $\overline{C_1 \oplus \dotsc \oplus C_t}$ is an irreducible component of $\Rep_{\C Q/I}(\beta)$ if and only if $$\operatorname{ext}^1_{\C Q/I}(C_i, C_j)
=0$$ for all $i\neq j$.  
 \end{theorem}
 \begin{corollary}\label{cor:bandcomponents}
 In general, $\bigcup\limits_{\lambda \in (\C^*)^{B(\Gamma)}} \mathcal{O}_{M_\lambda}$ is dense in $\Rep_{Q, c}(\beta, r)$.
 \end{corollary}
\begin{proof}
Enumerate the connected components of $\Gamma_{Q, c}(\beta, r, \epsilon)$,  $c_1,\dotsc, c_t$ with $c_i$ a band for $i=1,\dotsc, l$ and a string for $i=l+1,\dotsc, t$.  Let $\beta\vert_i, r\vert_i$ be the restrictions of $\beta$ and $r$, respectively, to the $i$-th connected component.  (By proposition \ref{note:UDGraph}, each connected component is itself an up-and-down graph, so is associated with a dimension vector and maximal rank map.)  Let $C_i = \Rep_{Q, c}(\beta|_i, r|_i)$, which is an irreducible component by \ref{prop:irrcomponents}.  Notice that $M_{Q, c}(\beta|_i, r|_i) \in \Rep_{Q, c}(\beta|_i, r|_i)$ if $c_i$ is a string and $M_{Q, c}(\beta|_i, r|_i)_{\lambda_i} \in \Rep_{Q, c}(\beta|_i, r|_i)$ if $c_i$ is a band.  Thus, the $C_i$ are irreducible and, assuming theorem \ref{thm:vanishingextbands} is true, $\operatorname{ext}^1_{\C Q/I_c}(C_i, C_j) = 0$, so $\Rep_{Q, c}(\beta, r) = \overline{C_1 \oplus \dotsc \oplus C_t}$.  

Thus, all that remains to be shown is that if $c_i$ is a band, then the union of the orbits of all elements in $M_{Q, c}(\beta|_i, r|_i)$ contains an open set.  Indeed, if this is the case, then denoting by $S_i$ the set $\GL(\beta|_i)\cdot M_{Q, c}(\beta|_i, r|_i)$ we have $$\overline{C_1 \oplus \dotsc \oplus C_t} = \overline{S_1 \oplus \dotsc \oplus S_t}.$$

Suppose that $\beta$ is a dimension vector and $r$ is a maximal rank map such that $\Gamma_{Q, c}(\beta, r, \epsilon)$ is a single band.  Let $M_\lambda=M_{Q, c}(\beta, r)_\lambda$, and denote by $\mathcal{O}_{M_\lambda}$ the $\GL(\beta)$-orbit of $M_\lambda$.  From Kraft (2.7 \cite{Kraft2}), there is an embedding 
\[
T_{M_\lambda}(\Rep_{Q, c}(\beta, r))/T_{M_\lambda}(\mathcal{O}_{M_\lambda}) \hookrightarrow \Ext^1(M_\lambda, M_\lambda)
\] 
where $T_M(X)$ denotes the tangent space in $X$ at $M$. 
By theorem \ref{thm:vanishingextbands}, then 
\[
\dim T_{M_\lambda}(\Rep_{Q, c}(\beta, r)) - \dim T_{M_\lambda}(\mathcal{O}_{M_\lambda}) \leq 1.
\]

\begin{claim}
 $M_\lambda$ is a non-singular point in $\Rep_{Q, c}(\beta, r)$ (and in $\mathcal{O}_{M_\lambda}$).  
\end{claim}
\begin{proof}Consider the construction of $M_{Q, c}(\beta, r)$ as a specific choice of embedding a product of varieties of complexes into $\Rep_{Q, c}(\beta, r)$.  Namely, for each color $s\in S$, we can define $\operatorname{Com}(\beta, r, s)$ to be the variety of representations $V$ of $(Q, c)$ with $\dim V_y=0$ if $(y,s)\notin \X$, $\dim V_x = \beta_x$, and $\rank V_a \leq r(a)$ whenever $c(a)=s$ (clearly zero otherwise).  This is the variety of vector space complexes, and by \cite{CW}, we have that 
\begin{align*}
 \prod\limits_{s\in S} \operatorname{Com}(\beta, r, s) \cong \Rep_{Q, c}(\beta, r)
\end{align*}
For each $x\in Q_0$, let $\sigma(x)$ be the matrix of the map (in the distinguished basis of $M_{Q, c}(\beta, r)$) corresponding to the permutation $(1, \beta_x)(2,\beta_x-1)\dotsc$.  If $V\in \prod_{s\in S}\operatorname{Com}(\beta, r, s)$, then define by $\varphi(V)$ the element of $\Rep_{Q, c}(\beta, r)$ with 
\begin{align*}
\varphi(V)_a = \begin{cases}V_a & \textrm{ if } \epsilon(ta, c(a))=1=-\epsilon(ha, c(a))\\ \sigma(ha)V_a & \textrm{ if } \epsilon(ta, c(a))=1=\epsilon(ha, c(a))\\ V_a \sigma(ta) & \textrm{ if } \epsilon(ta,c(a))=-1=\epsilon(ha,c(a))\\ \sigma(ha) V_a \sigma(ta) & \textrm{ if } \epsilon(ta, c(a))=-1=-\epsilon(ha, c(a))\end{cases}
\end{align*}
The map $\varphi$ is an isomorphism, since $\sigma(x)\in \GL(\beta_x)$ for each $x$.  Furthermore, $\rank V_a = \rank \varphi(V)_a$.  Therefore, since $\rank (M_{Q, c}(\beta, r)_a) = r(a)$, there is a $V\in \prod\limits_{s\in S}\operatorname{Com}(\beta, r, s)$ with $\varphi(V)=M_{Q, c}(\beta, r)$, and $\rank V_a=r(a)$.  It is shown in \cite{CW} that the variety $\operatorname{Com}(\beta, r, s)$ has a dense open orbit, given by the those complexes $W$ such that $\rank W_a=r(a)$.  Thus, $\prod\limits_{s\in S}\operatorname{Com}(\beta, r, s)$ is smooth at $V$, and therefore $\Rep_{Q, c}(\beta, r)$ is smooth at $M_{Q, c}(\beta, r)$.  \end{proof}

Hence, we have the following: 
\[
	\dim (\Rep_{Q, c}(\beta, r)) - \dim (\mathcal{O}_{M_\lambda})=\dim T_{M_\lambda}(\Rep_{Q, c}(\beta, r)) - \dim T_{M_\lambda}(\mathcal{O}_{M_\lambda})\leq 1.
\]
If the difference is 0, then $\overline{\mathcal{O}_{M_\lambda}}$ is a closed set of the same dimension as $\dim (\Rep_{Q, c}(\beta, r))$, so these are equal.  On the other hand, if the difference is 1, then $X:=\overline{\bigcup\limits_{\lambda\in \C^*} \mathcal{O}_{M_\lambda}}$ is a closed set.  For $t\in \C^*$, $M_{\lambda+t}\not\cong M_\lambda$ and $M_{\lambda+t} \in X$.  Therefore, $\dim T_{M_\lambda} X \geq \dim T_{M_\lambda} \mathcal{O}_{M_\lambda} + 1$, and therefore $\dim T_{M_\lambda} X = \dim (\Rep_{Q, c}(\beta, r))$.  Since $X$ is closed, $X=\Rep_{Q, c}(\beta, r)$.  
 \end{proof}
 
 In order to prove theorem \ref{thm:vanishingextbands}, we will explicitly describe the projective resolution of $M_{Q, c}(\beta,r)_\lambda$ for any $\beta, r$, and then apply the appropriate $\Hom$-functor to the resolution. 

%%%%%%%%%%%%%%%%%%THE PROJECTIVE RESOLUTION%%%%%%%%%
\subsection{Projective resolutions of $M_{Q, c}(\beta, r)$ and the $\EXT$-graph}
The summands in the projective resolutions of $M_{Q, c}(\beta, r)$ depend on a number of characteristics of the graph $\Gamma_{Q, c}(\beta, r, \epsilon)$.  We collect the pertinent characteristics in the following list.
\begin{definition}
Let $\Gamma=\Gamma_{Q, c}(\beta, r, \epsilon)$ be a fixed up-and-down graph.  
\begin{itemize}
\item[a.] Denote by $\ISO(\Gamma)$ the set of isolated vertices in $\Gamma$;
\item[b.] Denote by $\LEND(\Gamma)$ (resp. $\REND(\Gamma)$) the set of sources (resp. targets) of degree one in $\Gamma$.  These will be referred to as 1-sources (resp. 1-targets).
\item[c.] For a vertex $v_j^x\in T(\Gamma)$, we denote by $lp^+(v_j^x)$ (resp. $lp^-(v_j^x)$) the longest left positive (resp. left negative) direct path in $\Gamma$ terminating in $v_j^x$ (if such a path exists).  Similarly, for a vertex $v_j^x \in S(\Gamma)$, denote by $rp^+(v_j^x)$ (resp. $rp^-(v_j^x)$) the longest right positive (resp. right negative) direct path initiating in $v_j^x$.
\item[d.] For a vertex $v_j^x \in T(\Gamma)$, we denote by $l^+(v_j^x)$, (resp. $l^-(v_j^x)$) the source at the other end of $lp^+(v_j^x)$ (resp. $lp^-(v_j^x)$).  Similarly, for a vertex $v_j^x \in S(\Gamma)$, we denote by $r^+(v_j^x)$ (resp. $r^-(v_j^x)$) the target at the other end of $rp^+(v_j^x)$ (resp. $rp^-(v_j^x)$).
\item[e.] If $v_j^x\in \LEND\cup\REND$, let $p_j^x$ be the direct path of maximal length containing $v_j^x$;
\item[f.] If $v_j^x \in \LEND\cup\REND$, denote by $[v_j^x]_1\in Q_1$ be the arrow with the property that $t([v_j^x]_1) = x$ and $c([v_j^x]_1)=c(w(e))$ where $e$ is the edge in $p_j^x$ containing $v_j^x$;
\item[g.] Furthermore, recursively define the arrows $[v_j^x]_l$ with $t([v_j^x]_l) = h([v_j^x]_{l-1})$, and $c([v_j^x]_l) = c([v_j^x]_1)$.  
\item[h.] Suppose $v_j^x\in \ISO$.  Denote by $[v_j^x]_1^+$ (resp. $[v_j^x]_1^-$) the arrow (if such exists) with $t([v_j^x]_1^\pm)=x$ and $\epsilon(x,c([v_j^x]_1^\delta))=\delta$.  Again, recursively define $[v_j^x]_l^\delta$ with $t([v_j^x]_l^\delta)=h([v_j^x]_{l-1}^\delta)$, and $c([v_j^x]_l^\delta)=c([v_j^x]_1^\delta)$.
\item[i.] In case $[v_j^x]_l$ or $[v_j^x]_l^\pm$ fails to exist, write $h([v_j^x]_l):=\emptyset$ (or $h([v_j^x]_l^\pm):=\emptyset$), and let $P_\emptyset$ be the zero object.  (This is nothing more than notation to write the projective resolution of up-and-down modules in a more compact form.)
\end{itemize}
\end{definition}

\begin{example} Referring again to example \ref{example1}, we have the following aspects:
\begin{itemize}
\item[i.] $\ISO(\Gamma)=\emptyset$;
\item[ii.] $v_3^{(1)}$ is a 1-source, and $p_3^{(1)}$ is the path $v_2^{(6)} e_2 v_1^{(2)} e_1 v_3^{(1)}$ where $w(e_1)=r_1$ and $w(e_2)=p_2$;
\item[iii.] $r^+(v_1^{(1)}) = v_1^{(6)}$, and $r^-(v_1^{(1)}) = v_3^{(2)}$. 
\item[iv.] Since $\epsilon((6), b_2)=-1$, we have $lp^-(v_1^{(6)})=b_2g_1$.  Similarly, $rp^-(v_1^{(1)})=b_2g_1$.  
\end{itemize}
To illustrate the situation (e)-(h), consider the dimension vector and rank sequence below: 
\[\xymatrix@C=60pt@R=1.7ex{
*+[F]{1} \ar@[|(4)]@*{[red]}[r]|*+[o][F]{1} \ar@[|(4)]@*{[green]}[dddr]|<<<<<*+[o][F]{0} & *+[F]{2} \ar@[|(4)]@*{[red]}[r]|*+[o][F]{0} \ar@[|(4)]@*{[pink]}[dddr]|<<<<*+[o][F]{0}& *+[F]{0} \\
&&\\
&&\\
*+[F]{0} \ar@[|(4)]@*{[blue]}[r]|*+[o][F]{0} \ar@[|(4)]@*{[pink]}[uuur]|<<<<*+[o][F]{0} & *+[F]{0} \ar@[|(4)]@*{[green]}[uuur]|<<<<*+[o][F]{0} \ar@[|(4)]@*{[blue]}[r]|*+[o][F]{0} & *+[F]{0} }\]
The associated up-and-down graph is given by \[\xymatrix@1@R=.5ex{v_1^{(1)} \ar@{-}[r]^{r_1} & v_1^{(2)} \\ & v_2^{(2)}}\]
In this case, $v_1^{(1)}\in S^1$, and $[v_1^{(1)}]_1=r_2$ since the longest path containing $v_1^{(1)}$ is $v_1^{(2)} r_1 v_1^{(1)}$, $c(r_1)=c(r_2)$, and $t(r_2)=(2)$.  The vertex $v_2^{(2)}$ is isolated, and in this case, $[v_2^{(2)}]_1^+=p_2$ and $[v_2^{(2)}]_1^-=r_2$.   
\end{example}

We are now prepared to exhibit the projective resolution in the general case.  Notice that the simple factor modules of $M_{Q, c}(\beta, r, \epsilon, \Theta)_\lambda$ are $S_x$ for $v_j^x \in S(\Gamma)$.
\begin{proposition}\label{prop:projectiveresolution}
 The following is a projective resolution of $M_{Q, c}(\beta, r, \epsilon, \Theta)_\lambda$ is:
\begin{align*}
\xymatrix{
\dotsc \ar[r] & P(M_\lambda)_2 \ar[r]^{\delta_{(M_\lambda)_1}} & P(M_\lambda)_1 \ar[r]^{\delta_{(M_\lambda)_0}} & P(M_\lambda)_0 \ar[r] & M_\lambda \ar[r] & 0}
\end{align*}
where 
\begin{align*}
P(M_\lambda)_0&= \bigoplus\limits_{v_j^x \in S(\Gamma)} P_x\\
P(M_\lambda)_1&= \bigoplus\limits_{v_i^y \in \RCOR} P_y \oplus\bigoplus\limits_{v_j^x \in \LEND \cup\REND} P_{h([v_j^x]_1)}\oplus \bigoplus\limits_{v_j^x \in \ISO}P_{h([v_j^x]_1^+)}\oplus P_{h([v_j^x]_1^-)}\\
P(M_\lambda)_l &= \bigoplus\limits_{v_j^x \in \REND\cup \LEND} P_{h([v_j^x]_l)} \oplus \bigoplus\limits_{v_j^x \in \ISO} P_{h([v_j^x]_l^+)} \oplus P_{h([v_j^x]_l^-)}; 
\end{align*}
and where the differential is given by the following maps (we write $P_{x,j}$ for the projective $P_x$ arising from $v_x^j$):
\begin{itemize}
\item[i.] If $v_i^y\in \RCOR$, $v_{i^+}^{y^+} = l^+(v_i^y)$ and $v_{i^-}^{y^-} =l^-(v_i^y)$, then the map $\delta(M)_0$ restricts to $$\xymatrix@C=20ex{P_{y,i} \ar[r]^{\begin{bmatrix} lp^+(v_i^y) \\ -\lambda_blp^-(v_i^y) \end{bmatrix}} & P_{y^+, i^+} \oplus P_{y^-, i^-}}$$ if $v_i^y = \Theta(b)$ for some band $b$, and $$\xymatrix@C=20ex{P_{y,i} \ar[r]^{\begin{bmatrix} \phantom{\lambda_b} lp^+(v_i^y) \\ - lp^-(v_i^y) \end{bmatrix}} &P_{y^+, i^+} \oplus P_{y^-, i^-} }$$ otherwise.
\item[ii.] If $v_i^y \in \REND$, $p_i^y$ is the longest direct path terminating at $v_i^y$, and $v_j^x$ is the source at the other end of $p_i^y$, then the restriction of $\delta(M)_0$ to $P_{h([v_i^y]_1)}$ is given by $$\xymatrix@C=20ex{ P_{h([v_i^y]_1)} \ar[r]^{\begin{bmatrix} [v_i^y]_1A(p_i^y) \end{bmatrix}}& P_{x,j}}.$$
\item[iii.] If $v_i^y \in \ISO$, then restriction of $\delta(M)_0$ to $P_{h([v_i^y]_1^+)} \oplus P_{h([v_i^y]_1^-)}$ is given by $$\xymatrix@C=20ex{P_{h([v_i^y]_1^+)} \oplus P_{h([v_i^y]_1^-)} \ar[r]^{\begin{bmatrix} [v_i^y]_1^+ \quad [v_i^y]_1^- \end{bmatrix}} & P_{y,i}}.$$
\item[iv.] If $v_i^y \in \REND \cup \LEND$, then the restriction of $\delta(M)_l$ to $P_{h([v_i^y]_{l+1})}$ is $$\xymatrix@C=20ex{P_{h([v_i^y]_{l+1})} \ar[r]^{\begin{bmatrix} [v_i^y]_l \end{bmatrix}} & P_{h([v_i^y]_l)}}.$$
\item[v.] If $v_i^y \in \ISO$, then $\delta(M)_l$ restricted to $P_{h([v_i^y]_{l+1}^{\pm})}$ is $$\xymatrix@C=20ex{P_{h([v_i^y]_{l+1}^{\pm})} \ar[r]^{\begin{bmatrix} [v_i^y]_l^{\pm}\end{bmatrix}} & P_{h([v_i^y]_l^{\pm})}}.$$
\end{itemize}
\end{proposition}
We now apply the functor $\Hom(-, M_\mu)$ to the complex $P(M_\lambda)_\bullet$.  Recall that we have a fixed basis for the spaces $(M_\mu)_x$ for each $x\in Q_0$, namely $\{e_1^x, \dotsc, e_{\beta_x}^x\}$, relative to which the arrows act by the description given by the graph $\Gamma_{Q, c}(\beta, r, \epsilon)$.  So we take $\{v_i^x \boxtimes e_j^x\}_{j=1,\dotsc, \beta_x}$ the basis for $\Hom(P_{x,i}, M_\mu)$, $\{v_i^x \boxtimes e_j^{h([v_i^x]_l)}\}_{j=1,\dotsc, \beta_{h([v_i^x]_l)}}$ the basis for $\Hom(P_{h([v_i^x]_l)}, M_\mu)$ for $v_i^x \in \LEND \cup \REND$, and $\{v_i^x \boxtimes e_j^{h([v_i^x]_l^t)}\}_{j=1,\dotsc, \beta_{h([v_i^x]_l^t)}}$ for $v_i^x \in \ISO$ and $t=+, -$, relative to the aforementioned bases.  

%\begin{align*}
%\xymatrix{
%\bigoplus\limits_{v_j^x \in S(\Gamma)} (M_\mu)_x \ar[r] & \bigoplus\limits_{v_j^x \in \REND\cup \LEND} (M_\mu)_{h([v_j^x]_1)} \oplus \bigoplus\limits_{v_j^x \in \RCOR} (M_\mu)_x  \ar[r] & \bigoplus\limits_{v_j^x \in \REND\cup \LEND} (M_\mu)_{h([v_j^x]_2)} \ar[r] & \dotsc }
%\end{align*}

We will construct a graph $\mathbb{EXT}$ whose vertices correspond to a fixed basis for $\Hom(P(M_\lambda)_\bullet, M_\mu)$ as described above.  We will partition the vertices into subsets $\mathbb{EXT}(i)$ for $i=0, 1,\dotsc$ called \emph{levels}.  From this graph the homology of the complex can be easily read.  
\begin{definition}
Let $M_\lambda$ be as described above.  Let $\mathbb{EXT}(l)$ be the sets defined as follows.
\begin{align*}
\mathbb{EXT}(0) &= \{v_j^x\boxtimes v_{j'}^x \}_{\substack{v_j^x \in S(\Gamma)\\ j'=1,\dotsc, \beta_x}}\\
\mathbb{EXT}(1) &= \{v_j^x \boxtimes v_{j'}^x\}_{\substack{v_j^x \in \RCOR \\ j'=1,\dotsc, \beta_x}} \cup \{v_j^x \boxtimes v_{j'}^{h([v_j^x]_1)}\}_{\substack{v_j^x \in \REND \cup \LEND\\ j'=1,\dotsc, \beta_{h([v_j^x]_1)}}} \cup \{v_j^x \boxtimes v_{j'}^{h([v_j^x]_1^{t})}\}_{\substack{v_j^x \in \ISO\\ j'=1,\dotsc, \beta_{h([v_j^x]_1^t)} \\ t=+, -}}\\
\mathbb{EXT}(l) &=\{v_j^x \boxtimes v_{j'}^{h([v_j^x]_l)}\}_{\substack{v_j^x \in \REND \cup \LEND \\ j'=1,\dotsc, h([v_j^x]_l)}}\cup \{v_j^x \boxtimes v_{j'}^{h([v_j^x]_l^{t})}\}_{\substack{v_j^x \in \ISO\\ j'=1,\dotsc, \beta_{h([v_j^x]_l^t)} \\ t=+, -}}
\end{align*}
and $\mathbb{EXT}$ the graph with vertices $\bigcup\limits_{l\geq0} \mathbb{EXT}(l)$ and edges given by 
\begin{itemize}
\item[a.] $\xymatrix{v_j^x \boxtimes v_{j'}^x \ar@{-}[r] & v_i^y \boxtimes v_{i'}^{y'}}$ if $$\Hom(\delta(M_\lambda)_0, M_\mu)): v_j^x \boxtimes e_{j'}^x \mapsto \sum s_{j,j',x}^{i,i',y,y'} v_i^y \boxtimes e_{i'}^{y'}$$ with $s_{j,j',x}^{i,i',y,y'}\neq 0$ between levels $\mathbb{EXT}(0)$ and $\mathbb{EXT}(1)$;
\item[b.] $\xymatrix{v_j^x \boxtimes v_{j'}^{x'} \ar@{-}[r] & v_i^y \boxtimes v_{i'}^{y'}}$ if $$\Hom(\delta(M_\lambda)_l, M_\mu)): v_j^x \boxtimes e_{j'}^{x'} \mapsto \sum s_{j,j',x,x'}^{i,i',y,y'} v_i^y \boxtimes e_{i'}^{y'}$$ and $s_{i,y,i',y'}\neq 0$ between $\mathbb{EXT}(l-1)$ and $\mathbb{EXT}(l)$.
\end{itemize}
\end{definition}

\subsection{Properties of the $\EXT$-graph} We collect now the properties of the $\EXT$ graph that will be used to show exactness of complex $\Hom(P(M_\lambda)_\bullet, M_\mu)$.  
\begin{proposition}\label{prop:EXTPROPERTIES} Let $\mathbb{EXT}$ be the graph given above 
\begin{itemize}
\item[E1.] There is an edge $$\xymatrix{\mathbb{EXT}(0)\ni v_j^x \boxtimes v_{j'}^x \ar@{-}[r] & v_i^y \boxtimes v_{i'}^y \in \mathbb{EXT}(1)}$$ in the graph $\mathbb{EXT}$ if $v_j^x \in \LCOR$, $v_i^y \in \RCOR$, $\xymatrix{v_j^x \ar@{-}[r]^p & v_i^y}$ and $\xymatrix{v_{j'}^x \ar@{-}[r]^{p'} & v_{i'}^y}$ are paths in $\Gamma$ with $A(p)=A(p')$.
\item[E2.] If $v_i^y \in \REND$, $v_j^x =l^{\pm}(v_i^y) \in S(\Gamma)$ and $p=lp^{\pm}(v_i^y)$, then there is an edge $$\xymatrix{\mathbb{EXT}(0) \ni v_j^x \boxtimes v_{j'}^x \ar@{-}[r] & v_i^y \boxtimes v_{i'}^{y'} \in \mathbb{EXT}(1)}$$ if $\xymatrix{v_{j'}^x \ar@{-}[r]^{p'} & v_{i'}^{y'}}$ is a path in $\Gamma$ with $A(p') = [v_i^y]_1 A(p)$.  Furthermore, there is an edge $$\xymatrix{\mathbb{EXT}(l) \ni v_i^y \boxtimes v_{i'}^{h([v_i^y]_l)} \ar@{-}[r] & v_i^y \boxtimes v_{j'}^{h([v_i^y]_{l+1})} \in \mathbb{EXT}(l+1)}$$ in $\mathbb{EXT}$ if there is an edge $\xymatrix{v_{i'}^{h([v_i^y]_l)} \ar@{-}[r]^e & v_{j'}^{h([v_i^y]_{l+1})}}$ in $\Gamma$ with $w(e)=[v_i^y]_{l+1}$.  
\item[E3.] Similarly, if $v_i^y \in \LEND$, then there is an edge $$\xymatrix{\mathbb{EXT}(0)\ni v_i^y \boxtimes v_{i'}^y \ar@{-}[r] & v_i^y \boxtimes v_{j'}^{h([v_i^y]_1)} \in \mathbb{EXT}(1)}$$ in $\mathbb{EXT}$ if there is an edge $\xymatrix{v_{i'}^y \ar@{-}[r]^e & v_{j'}^{h([v_i^y]_1)}}$ with $w(e)=[v_i^y]_1$.  Furthermore, there is an edge $\xymatrix{\mathbb{EXT}(l-1) \ni v_j^x \boxtimes v_{j'}^{h([v_j^x]_{l-1})} \ar@{-}[r] & v_j^x \boxtimes v_{j''}^{h([v_j^x]_l)} \in \mathbb{EXT}(l)}$ in $\mathbb{EXT}$ if there is an edge $\xymatrix{v_{j'}^{h([v_j^x]_{l-1})} \ar@{-}[r]^e & v_{j''}^{h([v_j^x]_l)}}$ in $\Gamma$ with $w(e)=[v_j^x]_l$.
\item[E4.] Finally, if $v_i^y\in \ISO$, then there is an edge $$\xymatrix{\mathbb{EXT}(0)\ni v_i^y\boxtimes v_{i'}^y \ar@{-}[r] & v_i^y \boxtimes v_j^{h([v_i^y]_1^{\pm})} \in \mathbb{EXT}(1)}$$ in $\mathbb{EXT}$ if there is an edge $\xymatrix{v_{i'}^y \ar@{-}[r]^e & v_j^{h([v_i^y]_1^{\pm})}}$ in $\Gamma$ with $w(e)=[v_i^y]_1^{\pm}$.  Furthermore, there is an edge $$\xymatrix{\mathbb{EXT}(l-1)\ni v_i^y\boxtimes v_{i'}^{h([v_i^y]_{l-1}^{\pm})} \ar@{-}[r] & v_i^y \boxtimes v_j^{h([v_i^y]_l^{\pm})} \in \mathbb{EXT}(l)}$$ in $\mathbb{EXT}$ if there is an edge $\xymatrix{v_{i'}^{h([v_i^y]_{l-1}^{\pm})} \ar@{-}[r]^e & v_{j}^{h([v_i^y]_l^{\pm})}}$ in $\Gamma$ with $w(e)=[v_i^y]_l^{\pm}$.  
\end{itemize}
\end{proposition}
\begin{lemma}\label{lem:noisolatedvertices}
There are no isolated vertices in $\mathbb{EXT}(1)$.  
\end{lemma}
\begin{proof}

First, suppose $v_i^y \boxtimes v_{i'}^y \in \mathbb{EXT}(1)$ (i.e., $v_i^y \in \RCOR$).  If $i'<i$ (resp. $i'>i$), then by lemma \ref{lem:rightleftconnection}, there is a path $p'$ terminating at $v_{i'}^y$ with $A(p')=lp^-(v_i^y)$ (resp. $A(p')=lp^+(v_i^y)$).  Therefore, there is an edge $\xymatrix{v_j^x \boxtimes v_{j'}^x \ar@{-}[r] & v_i^y \boxtimes v_{i'}^y}$.

Next, suppose $v_i^y \in \REND$, and $[v_i^y]_1$ exists (otherwise, no vertex $v_i^y\boxtimes v_{i'}^{y'}$ would exist in $\Gamma$).  Let $p$ be the path of maximal length terminating at $v_i^y$, and  $v_j^x$ the source at which $p$ starts.  Label the edge of $p$ containing $y$ by $a_1$, let $b_2:=[v_j^x]_1$, and $b_1$ the arrow (if it exists) with $h(b_1)=y$ and $c(b_1)=c(b_2)$.  By lemma \ref{note:corners}, $r(b_1)+r(b_2)<\beta_y$.  Now denote by $b_3$ the arrow $[v_i^y]_2$.  By lemma \ref{lem:maximalrank}, $r(b_2)+r(b_3)=\beta_{hb_2}$, so $v_{i'}^{hb_2}$ is contained in an edge with such a label.  If said label is $b_2$, then $e_{i'}^{y'}\in \im b_2A(p)$, and so $v_i^y \boxtimes v_{i'}^{hb_2}$ is contained in an edge between $\EXT(1)$ and $\EXT(0)$.  Otherwise, $b_3 e_{i'}^{hb_2}=e_{i''}^{hb_3}\neq 0$.  In this case, $v_i^y \boxtimes v_{i'}^{hb_2}\in \EXT(1)$ and $v_i^y\boxtimes v_{i''}^{hb_3} \in \EXT(2)$ are contained in an edge. 

Finally, suppose that $v_i^y\in \ISO$, and let $y'=h([v_i^y]_1^+)$ or $h([v_i^y]_1^-$.  We will show that $v_i^y \boxtimes v_{i'}^{y'}$ is non-isolated for $i=1,\dotsc, \beta_{y'}$.  Note first that $v_{i'}^{y'}$ is non-isolated in $\Gamma$ by lemma \ref{lem:maximalrank}, for suppose that $a_0$ is the arrow (if it exists) with $h(a_0)=y$, and $c(a_0)=c([v_i^y]_1^{\pm})$.  By lemma \ref{note:corners}, $r(a_0)+r([v_i^y]_1^{\pm})< \beta_y$, so by lemma \ref{lem:maximalrank}, $r([v_i^y]_1^{\pm})+r([v_i^y]_2^{\pm}) = \beta_{y'}$.  Therefore, there is an edge $e$ incident to $v_{i'}^{y'}$ such that $w(e)=[v_i^y]_1^{\pm}$ or $[v_i^y]_2^{\pm}$.  In the former case, $v_i^y \boxtimes v_{i'}^{y'}$ is contained in a common edge with a vertex in $\mathbb{EXT}(0)$, and in the latter case it is contained in a common edge with a vertex in $\EXT(2)$.  
\end{proof} 
\begin{lemma}\label{lem:stringandband}
All vertices in $\EXT$ are contained in at most two edges, and every vertex with label $v_i^y\boxtimes v_{i'}^{h([v_i^y]_l}$ for $l\geq 1$ is contained in at most one edge.  Furthermore, the neighbor of any vertex $v_i^y \boxtimes v_{i'}^{h([v_i^y]_l)}$ in $\mathbb{EXT}(l)$ is $v_i^y \boxtimes v_{i''}^{h([v_i^y]_{l-1})}$ or $v_i^y \boxtimes v_{i''}^{h([v_i^y]_{l+1})}$ for some $i''$.  Therefore, the graph $\mathbb{EXT}$ splits into string and band components, such that the band components and strings of length greater than one occur between levels $\mathbb{EXT}(0)$ and $\mathbb{EXT}(1)$. 
\end{lemma}
\begin{proof}
Recall from property E2 that $v_i^y \boxtimes v_{i'}^{h([v_i^y]_1)}$ is connected by an edge to $v_j^x \boxtimes v_{j'}^x \in \mathbb{EXT}(0)$ if and only if $v_i^y \in \REND$, $\xymatrix{v_j^x \ar@{~}[r]^p & v_i^y}$ is the longest left direct path in $\Gamma$ ending at $v_i^y$, and there is a path $\xymatrix{v_{j'}^x \ar@{~}[r]^{p'} & v_{i'}^{h([v_i^y]_1)}}$ with $A(p') = [v_i^y]_1 A(p)$.  It is clear that there is only one such vertex, if it exists.  If such a path does exist, then there is no edge in $\mathbb{EXT}$ between $v_i^y \boxtimes v_{i'}^{h([v_i^y]_1)}$ and $v_i^y \boxtimes v_{i''}^{h([v_i^y]_2)}$, since this would mean that $v_{i'}^{h([v_i^y]_1)}$ and $v_{i''}^{h([v_i^y]_2)}$ are contained in an edge $e$ in $\Gamma$ with $w(e)=[v_i^y]_2$.  This contradicts proposition \ref{note:UDGraph}, since $v_{i'}^{h([v_i^y]_1)}$ would be in two edges of the same color.  Otherwise, $v_i^y \boxtimes v_{i'}^{h([v_i^y]_1)}$ is connected to the vertex $v_i^y \boxtimes v_{i''}^{h([v_i^y]_2)}$ in $\mathbb{EXT}$ if and only if there is an edge $\xymatrix{v_{i'}^{h([v_i^y]_1)}\ar@{-}[r]^e & v_{i''}^{h([v_i^y]_2)}}$ with $w(e)=[v_i^y]_2$, by property E3.  By definition of the Up and Down graph, this describes a unique vertex. \\
As for the other vertices, the lemma is clear from property E1.
\end{proof}
In terms of the complex $M_\bullet$, the above lemma says that the kernel of the map $\Hom(\delta_2, M_\mu))$ is spanned by the elements $\{v_i^y \boxtimes v_{i'}^y\}_{\substack{v_i^y \in \RCOR\\ i'=1,\dotsc, \beta_y}}\cup \{v_i^y \boxtimes v_{i'}^{h([v_i^y]_1)}\}_{\substack{v_i^y \in \REND\cup \LEND \\ i'=1,\dotsc, \beta_{h([v_i^y]_1)}}}$ which share no edge with vertices in $\mathbb{EXT}(2)$. 
\begin{lemma}\label{lem:stringandbandtypes}
No string in $\mathbb{EXT}$ has both endpoints in $\mathbb{EXT}(1)$.
\end{lemma}
\begin{proof}
Suppose that there is a string with one endpoint $v_{j_0}^{y_0} \boxtimes v_{j_0'}^{y_0'} \in \mathbb{EXT}(1)$ and containing the following substring:
\begin{align*}
\xymatrix@R=3ex{
v_{i_1}^{x_1} \boxtimes v_{i_1'}^{x_1} \ar@{-}[r] \ar@{-}[dr]& v_{j_0}^{y_0} \boxtimes v_{j_0'}^{y_0'}\\
v_{i_2}^{x_2} \boxtimes v_{i_2'}^{x_2} \ar@{-}[r] \ar@{-}[dr]& v_{j_1}^{y_1} \boxtimes v_{j_1'}^{y_1'}\\
 \ar@{}[r]|{\vdots} \ar@{-}[dr]  &  \\
v_{i_n}^{x_n} \boxtimes v_{i_n'}^{x_n} \ar@{-}[r] \ar@{-}[dr] & v_{j_{n-1}}^{y_{n-1}} \boxtimes v_{j_{n-1}'}^{y_{n-1}'}\\
& v_{j_n}^{y_n} \boxtimes v_{j_n'}^{y_n'}}
\end{align*}
with $v_{i_t}^{x_t}\boxtimes v_{i_t'}^{x_t} \in \mathbb{EXT}(0)$ and $v_{j_s}^{y_s} \boxtimes v_{j_s'}^{y_s'} \in \mathbb{EXT}(1)$.  We will show that the string does not end in the vertex $v_{j_n}^{y_n}\boxtimes v_{j_n'}^{y_n'}$.  Recall by definition of the graph $\mathbb{EXT}$ that for such a string to exist, we must have paths 
\begin{align*}
\xymatrix@R=2ex{
v_{i_1}^{x_1} \ar@{~}[r]^{p_0} \ar@{~}[dr]^{q_1} & v_{j_0}^{y_0}\\
v_{i_2}^{x_2} \ar@{~}[r]^{p_1} \ar@{~}[dr]^{q_2} & v_{j_1}^{y_1}\\
\ar@{~}[dr] \ar@{}[r]|{\vdots} & \\
v_{i_n}^{x_n} \ar@{~}[r] \ar@{~}[dr]_{q_n} &v_{j_{n-1}}^{y_{n-1}} \\
& v_{j_n}^{y_n}}
\end{align*}
in $\Gamma$.  A small notational point: if $p$ is a direct path which starts (resp. ends) in the vertex $v_l^z$, with $e$ the edge of $p$ incident to said vertex, then we write $\epsilon(z,p):=\epsilon(z,c(w(e)))$.
\begin{itemize}  
\item[\underline{Case 1:}] Assume that $v_{j_0}^{y_0}, v_{i_n}^{y_n} \in \RCOR$.  Let $p_n$ be the longest left path terminating in $v_{i_n}^{y_n}$ with $p_n \neq q_n$ (this is guaranteed since $v_{i_n}^{y_n}$ is a 2-target).  Similarly, let $q_0$ be the longest left path terminating in $v_{j_0}^{y_0}$ with $q_0\neq p_0$.  
\begin{itemize}
\item[A:] If $i_0'<i_0$, then $\epsilon(y_0, p_0)=-1$.  If not, then by lemma \ref{lem:rightleftconnection} there would be a path $q_0'$ terminating at $v_{i_0}^{y_0}$ in $\Gamma$ with $A(q_0')=A(q_0)$.  By definition of the graph $\mathbb{EXT}$, then, there would be an other edge terminating at the vertex $v_{i_0}^{y_0} \boxtimes v_{i_0'}^{y_0}$.  
\begin{itemize}
\item[A1:] if $i_n'>i_n$, then $\epsilon(y_n, q_n)=-1$ by lemma \ref{lem:rightleftconnection}.  Thus, by proposition \ref{note:UDGraph}, $\epsilon(y_n, p_n)=1$.  Therefore, again by lemma \ref{lem:rightleftconnection}, there is a path $p_n'$ in $\Gamma$ terminating at $v_{i_n'}^{y_n}$ with $A(p_n')=A(p_n)$, so there is another edge in $\mathbb{EXT}$ containing the vertex $v_{i_n}^{y_n} \boxtimes v_{i_n'}^{y_n}$.  
\item[A2:] if $i_n'<i_n$, then $\epsilon(y_n, q_n)=1$ by lemma \ref{lem:rightleftconnection}.  Thus, by proposition \ref{note:UDGraph}, $\epsilon(y_n, p_n)=-1$.  Therefore, again by lemma \ref{lem:rightleftconnection}, there is a path $p_n'$ in $\Gamma$ terminating at $v_{i_n'}^{y_n}$ with $A(p_n')=A(p_n)$, so there is another edge in $\mathbb{EXT}$ containing the vertex $v_{i_n}^{y_n} \boxtimes v_{i_n'}^{y_n}$.  
\end{itemize}
\item[B:] If $i_0'>i_0$, then $\epsilon(y_0, p_0)=1$, by the same reasoning at Subcase A.  The subcases B1 and B2 are analogous to A1 and A2.
\end{itemize} 

\item[\underline{Case 2:}] Assume that $v_{j_0}^{y_0}\in \RCOR$ while $v_{i_n}^{y_n} \in \REND$.  We will show that Let $(i_n')^-$ be the integer such that there is an edge $\xymatrix{v_{(i_n')^-}^{y_n} \ar@{-}[r]^e & v_{i_n'}^{h([v_{i_n}^{y_n}]_1)}}$ in $\Gamma$ with $w(e)=[v_{i_n}^{y_n}]_1$.  This is guaranteed to exist by the definition of $[v_{i_n}^{y_n}]_1$ (refer to property E2 in proposition \ref{prop:EXTPROPERTIES}).  
\begin{itemize}
\item[A:] Suppose $(i_0)'<i_0$.  Then $\epsilon(y_0, p_0)=-1$ by definition of $\Gamma$.%write in lemma
\begin{itemize}
\item[A1:] If $(i_n')^-<i_n$, then $\epsilon(y_n,q_n)=1$, and since there is a path $eq_n$ in $\Gamma$ with $w(e)=[v_{i_n}^{y_n}]_1$, we must have that $\epsilon(y_n, [v_{i_n}^{y_n}]_1)=-1$.  If this were the case, then by the definition of the edges in $\Gamma$, there would be an edge $e'$ with $w(e')=[v_{i_n}^{y_n}]_1$ with one end at the vertex $v_{i_n}^{y_n}$.  This contradicts the assumption that $v_{i_n}^{y_n}$ is a 1-target.
\item[A2:] Similarly, if $i_n^- > i_n$, then $\epsilon(y_n, q_n)=1$, and since there is a path $eq_n$ in $\Gamma$ with $w(e)=[v_{i_n}^{y_n}]_1$, we have that $\epsilon(y_n, [v_{i_n}^{y_n}]_1)=-1$.  If this were the case, then there would be an edge $e'$ with $w(e')=[v_{i_n}^{y_n}]_1$ with one end at the vertex $v_{i_n}^{y_n}$, contradicting the assumption of $v_{i_n}^{y_n}$ being a 1-target.  
\end{itemize}
\item[B:] Suppose that $(i_0)'>i_0$.  Then $\epsilon(y_0,p_0)=1$ by definition of $\Gamma$.  Subcases b1 and b2 are the same as above with signs of $\epsilon$ flipped.
\end{itemize}

\item[\underline{Case 3:}] Assume that $v_{i_0}^{y_0}\in \REND$ and $v_{i_n}^{y_n}\in \RCOR$.  Let $p_n$ be the left direct path in $\Gamma$ of maximal length with endpoint $v_{i_n}^{y_n}$ and $p_n\neq q_n$ (guaranteed since the vertex is a 2-target).  As above, let $(i_0')^-$ be the integer such that there is an edge $e$ with endpoints $v_{(i_0')^-}^{y_0}$ and $v_{i_0'}^{h([v_{i_0}^{y_0}]_1)}$.
\begin{itemize}
\item[A:] Suppose that $(i_0')^-<i_0$.  Then $\epsilon(y_0, [v_{i_0}^{y_0}]_1)=1$, so $\epsilon(y_0,p_0)=-1$.  
\begin{itemize}
\item[A1:] If $i_n'<i_n$, then $\epsilon(y_n,q_n)=1$, so $\epsilon(y_n,p_n)=-1$.  But then by lemma \ref{lem:rightleftconnection}, there is an edge $p_n'$ with $A(p_n)=A(p_n')$ one of whose endpoints is $v_{i_n'}^{y_n}$.  
\item[A2:] If $i_n'>i_n$, then $\epsilon(y_n,q_n)=1$, $\epsilon(y_n,p_n)=1$.  By lemma \ref{lem:rightleftconnection}, there is an edge $p_n'$ with $A(p_n)=A(p_n')$ one of whose endpoints is $v_{i_n'}^{y_n}$.  
\end{itemize}
\item[B:] If $(i_0')^->i_0$, then the same arguments hold with the values of $\epsilon$ exchanged.
\end{itemize}

\item[\underline{Case 4:}] Assume that $v_{i_0}^{y_0}, v_{i_n}^{y_n} \in \REND$.  
\begin{itemize}
\item[A:] Suppose $(i_0')^-<i_0$, so $\epsilon(y_0, [v_{i_0}^{y_0}]_1) = -1$ and $\epsilon(y_0, p_0)=1$.
\begin{itemize}
\item[A1:] If $(i_n')^-<i_n$, then $\epsilon(y_n, q_n)=-1$ by lemma \ref{lem:rightleftconnection}.  But if this were the case, then there would be an edge $e$ in $\Gamma$ with $w(e)=[v_{i_n}^{y_n}]_1$ and one of whose endpoints was $v_{i_n}^{y_n}$.  This contradicts the assumption that said vertex was a 1-target.
\item[A2:] If $(i_n')^->i_n$, then $\epsilon(y_n,q_n) = 1$ by lemma \ref{lem:rightleftconnection}.  If this were the case, then there would be an edge $e$ in $\Gamma$ with $w(e)=[v_{i_n}^{y_n}]_1$ and one of whose endpoints was $v_{i_n}^{y_n}$.  This contradicts the assumption that said vertex was a 1-target.
\end{itemize}
\item[B:] If $(i_0')^-<i_0$, then the same argument holds with the values of $\epsilon$ exchanged.
\end{itemize}
\end{itemize}
\end{proof}

\subsection{Homology and the $\EXT$ graph}
Let us pause to interpret the above results into data concerning the maps $\Hom(\delta(M_\lambda)_1, M_\mu)$ and $\Hom(\delta(M_\lambda)_0, M_\mu)$.  Recall that a vertex $v_i^x\boxtimes v_j^y$ corresponds to the basis element $v_i^x \otimes e_j^y$.  By lemma \ref{lem:noisolatedvertices}, there are no isolated vertices in $\EXT(1)$, and by lemma \ref{lem:stringandband}, if $\Hom(\delta(M_\lambda)_1, M_\mu): v_i^x \otimes e_j^y \mapsto v_{i'}^{x'} \otimes e_{j'}^{y'}$, then after reordering the chosen basis, $\Hom(\delta(M_\lambda)_1, M_\mu)$ takes the form
\[
\begin{bmatrix} 
1 & 0 & \dotsc & 0 \\
0 & \ast & \dotsc & \ast\\
\vdots & \vdots & \ddots & \vdots\\
0 & \ast & \dotsc & \ast \end{bmatrix}
\]
In particular, $\ker(\Hom(\delta(M_\lambda)_1, M_\mu))$ is precisely the span of those vertices in $\EXT(1)$ that have an edge in common with a vertex in $\EXT(0)$.  

It remains to be shown that every other vertex in $\EXT(1)$ corresponds to a basis element that is in the image of $\Hom(\delta(M_\lambda)_0, M_\mu)$.  This will show that the image of said map equals the kernel of $\Hom(\delta(M_\lambda)_1, M_\mu)$.  Let us denote by $C_1, C_2,\dotsc, C_m$ the connected components of the induced subgraph on the vertices $\EXT(0)\cup \EXT(1)$.  Then $\Hom(\delta(M_\lambda)_1, M_\mu)$ can be written in block form:
\[ \begin{bmatrix} \delta_{C_1} & 0 &   &\dotsc& 0 \\ 0 & \delta_{C_2}  & \ddots & & \vdots \\ \vdots & & \ddots & \ddots & \\
0 & \dotsc & 0 & \delta_{C_m} & 0\end{bmatrix}\] Therefore, it suffices to show that each block corresponding to a connected component is surjective.

\begin{lemma}\label{lem:surjectivestrings} If $v_j^x \boxtimes v_i^y \in \EXT(1)$ is contained in a string between levels 0 and 1, then $v_j^x \otimes e_i^y \in \im(\Hom(\delta(M_\lambda)_0, M_\mu))$.  
\end{lemma}

\begin{proof}
Suppose that the vertex is contained in the connected component $C_i$, and that $C_i$ is a string.  We have shown in lemma \ref{lem:stringandbandtypes} that if a string is between levels 0 and 1, then either one endpoint lies in level 0 and the other in level 1, or  both endpoints lie in level 0.  In the first case, $\delta_{C_i}$ is strictly upper triangular with nonzero entries on the diagonal which must be from the set $\{\pm 1, \pm \lambda, \pm \mu\}$.  Therefore, the map is invertible.  In the second case, there is one more vertex in level $\EXT(0)$ than in $\EXT(1)$, and $(\delta_{C_i})_{j,j}\neq 0$ for each $j$, so the given map is surjective.
\end{proof}

\begin{lemma}\label{lem:bandexactness}
If $C_i$ is a band, then $\delta(C_i)$ is an isomorphism.
\end{lemma}
\begin{proof}
If a component $C_i$ is cyclic, then it must come from the following cycles on $\Gamma_{Q, c}(\beta, r, \epsilon)$:
\begin{align*}
\xymatrix{
v_{i_0}^{x_0} \ar@{~}[dddr] \ar@{~}[r]^{p_1} & v_{j_0}^{y_0}\\
v_{i_1}^{x_1} \ar@{~}[ur]^{q_1} \ar@{~}[r]^{p_2} & v_{j_1}^{y_1}\\
 \ar@{~}[ur]\ar@{}[r]|{\vdots} & \\
 v_{i_n}^{x_n} \ar@{~}[ur]^{q_{n-1}}\ar@{~}[r]_{p_n} & v_{j_n}^{y_n}} & \qquad &\xymatrix{
v_{i_0'}^{x_0} \ar@{~}[dddr] \ar@{~}[r]^{p_1} & v_{j_0'}^{y_0}\\
v_{i_1'}^{x_1} \ar@{~}[ur]^{q_1} \ar@{~}[r]^{p_2} & v_{j_1'}^{y_1}\\
 \ar@{~}[ur]\ar@{}[r]|{\vdots} & \\
 v_{i_n'}^{x_n} \ar@{~}[ur]^{q_{n-1}}\ar@{~}[r]_{p_n} & v_{j_n'}^{y_n}} 
\end{align*}
In particular, by definition of $\delta(M_\lambda)_1$, the matrix of $\delta(C_i)$ takes the following form:
\[
\begin{bmatrix}
1  & -\lambda & 0 & 0 & \dotsc & 0 \\
0 & \pm 1 & \pm 1 & 0& \dotsc& 0 \\ 
0 & 0 & \pm 1 & 0 & \dotsc & 0 \\
\vdots &  & \ddots &&& \vdots\\
\pm 1 & 0 & \dotsc & && \pm 1
\end{bmatrix} \]
where one of the diagonal entries is $\mu$, and in each row there is exactly one positive and one negative entry.  Then it is an elementary exercise (expanding by the first column and calculating the determinant of upper or lower triangular matrices) to show that $\det \delta(C_i)=\pm (\lambda-\mu)$.  Since, by assumption, $\lambda\neq \mu$, we have that $\delta(C_i)$ is nonsingular.  
\end{proof}
Now that part (a) of the theorem is proved, we move to part (b), recalled here:
\begin{proposition}
Suppose that $\Gamma_{Q, c}(\beta, r, \epsilon)$ consists of a single band component, and let $\lambda \in (\C^*)^{B(\Gamma)} = \C^*$.  Let $M_\lambda = M_{Q, c}(\beta, r, \epsilon)_\lambda$.  Then $$\Ext^1_{\C Q/I_c}(M_\lambda, M_\lambda) =1.$$
\end{proposition}
\begin{proof}
The projective dimension of $M_\lambda$ is one by the constructions above.  Furthermore, there is exactly one band component in the graph $\EXT$, since there is exactly one pair of bands $b_1, b_2$ in $\Gamma$ with the $A(p_i)=A(p_i')$ and $A(q_i) = A(q_i')$ as in the proof of lemma \ref{lem:bandexactness}.  Therefore, the image of the restriction of the map $\Hom(P(M_\lambda)_\bullet, M_\lambda)$ to the vectors $v_{i_k}^{x_k} \otimes e_{i_k}^{x_k}$ is in the span of the vectors $v_{j_k'}^{y_k} \otimes e_{j_k'}^{y_k}$.  Again, as in the proof of lemma \ref{lem:bandexactness}, the restriction of said map to the aforementioned subspaces relative to the basis given above is 
\[
C = \begin{bmatrix}
-\lambda		&			&			&			&		&			&			&		&1		\\
\pm\lambda			&\pm1			&			&			&		&			&			&		&		\\
			&\pm1			&	\pm1		&			&		&			&			&		&		\\
			&			&	\pm1		&	\pm1		&		&			&			&		&		\\
			&			&			&	\pm1		&\ddots	&			&			&		&		\\
			&			&			&			&		&  	\pm1		&			&		&		\\
			&			&			&			&		&	\pm1		&	1		&		&		\\
			&			&			&			&		&			&	\ddots	&\pm1		&		\\
			&			&			&			&		&			&			&\pm1		&\pm1			\end{bmatrix}.
			\]
Recall that in each row there is exactly one positive and one negative entry.  Therefore, the sum of the last $n-1$ columns of this matrix is $\begin{bmatrix} 1 \\ \pm1 \\ 0 \\ \vdots \\ 0\end{bmatrix}$ where the sign of the second entry is opposite of the sign of $\pm \lambda$.  Therefore, the first column is in the span of the last $n-1$ columns.  Column reducing gives the matrix 
\begin{align*}
C = \begin{bmatrix}
	0	&			&			&			&		&			&			&		&1		\\
			0&\pm1			&			&			&		&			&			&		&		\\
			&\pm1			&	\pm1		&			&		&			&			&		&		\\
			&			&	\pm1		&	\pm1		&		&			&			&		&		\\
			&			&			&	\pm1		&\ddots	&			&			&		&		\\
			&			&			&			&		&  	\pm1		&			&		&		\\
			&			&			&			&		&	\pm1		&	1		&		&		\\
			&			&			&			&		&			&	\ddots	&\pm1		&		\\
			&			&			&			&		&			&			&\pm1		&\pm1			\end{bmatrix}.
			\end{align*}
The lower right $n-1 \times n-1$ minor is clearly non-zero, since it is a strictly lower triangular matrix, so this map has rank $n-1$, showing that the complex $\Hom(P(M_\lambda)_\bullet, M_\lambda)$ has exactly one dimensional homology at $\Hom(P(M_\lambda)_1, M_\lambda)$.  
\end{proof}

\noindent{\bf Proof of theorem \ref{thm:vanishingextbands}}
By lemma \ref{lem:surjectivestrings}, blocks corresponding to strings on $\EXT$ are surjective, and by lemma \ref{lem:bandexactness}, blocks corresponding to bands on $\EXT$ are surjective, so the homology of the complex
\[ \xymatrix{ \Hom(P(M_\lambda)_0, M_\mu) \ar[r] & \Hom(P(M_\lambda)_1, M_\mu) \ar[r] & \dotsc }\]
vanishes in the first degree.  \begin{flushright} $\square$\end{flushright}

\section{Higher Extension Groups}
The graphical representation given above can be used to calculate higher extension groups.  For each vertex $v_j^x \in \LEND \cup \REND$, let $X_{j,x}$ be the complex 
\begin{align*}
\xymatrix@C=10ex{ M_x \ar[r]^{[v_j^x]_1)} & M_{h([v_j^x]_1)} \ar[r]^{[v_j^x]_2} & M_{h([v_j^x]_2} \ar[r]^{[v_j^x]_3} & \dotsc}.
\end{align*}
Furthermore, if $v_j^x \in \ISO$, let $X_{j,x}^+$ be the complex
\begin{align*}
\xymatrix@C=10ex{ M_x \ar[r]^{[v_j^x]^+_1} & M_{h([v_j^x]^+_1)} \ar[r]^{[v_j^x]^+_2} & \dotsc},
\end{align*}
and analogously for $X_{j,x}^-$.  Let $h^i(X)$ be the dimension of the $i$-th homology space of the complex $X$.
\begin{corollary}
Let $\Gamma_{Q, c}(\beta, r, \epsilon)$ be an up-and-down graph for $(Q, c)$ a gentle string algebra.  Then 
\begin{align*}
\dim \Ext^i(M_{Q, c}(\beta, r)_\lambda, M_{Q, c}(\beta, r)_\mu) &= \sum\limits_{v_j^x \in \LEND \cup \REND} h^i(X_{j,x}) + \sum\limits_{v_j^x \in \ISO}\left( h^i(X_{j,x}^+) + h^i(X_{j,x}^-)\right).
\end{align*}
\end{corollary}

\subsection{Example}
We finish by exhibiting the $\EXT$ graph for example \ref{example1}.  Recall that we chose $\Theta(b)=v_1^{(6)}$ for the band component.  By proposition \ref{prop:projectiveresolution}, the projective resolution of the representation in the example is given by 
\[
\xymatrix@C=15ex{M_\lambda & \ar[l] P_1^3 \oplus P_4^2  & P_2 \oplus P_3 \oplus P_5^2\oplus P_6 \ar[l]_{\delta_0}  & \ar[l]_{\delta_1} P_3}
\] 
where 
\[
\delta_0 = \begin{bmatrix} -r_1 & 0 &0&0 & -\lambda b_2g_1\\0 &0&-g_1 &0& p_2r_1\\ 0 & 0 & 0 & g_1&0\\  p_1 & -g_2b_1 &0 &0& 0\\ 0 & r_2p_1 & b_1 & 0&0 \end{bmatrix} \hspace{1in} 
\delta_1 = \begin{bmatrix} 0 \\ g_2 \\ 0 \\ 0 \\ 0 \end{bmatrix}\]
The associated $\EXT$ graph is obtained by applying $\Hom(-, M_\mu)$ to the resolution, so we have the complex:
\begin{align*}
\xymatrix@C=15ex{
(M_\mu)_1^3 \oplus (M_\mu)_4^2 \ar[r]^<<<<<<<<<<{\Hom(\delta_0, M_\mu)} & (M_\mu)_2 \oplus (M_\mu)_3 \oplus (M_\mu)^2_5 \oplus (M_\mu)_6 \ar[r]^<<<<<<<<<<<<<{\Hom(\delta_1, M_\mu)} & (M_\mu)_3}\end{align*}
The $\EXT$ graph is depicted below, with the vertices lying in a cyclic component of the graph boxed.  
\begin{align*}
\includegraphics{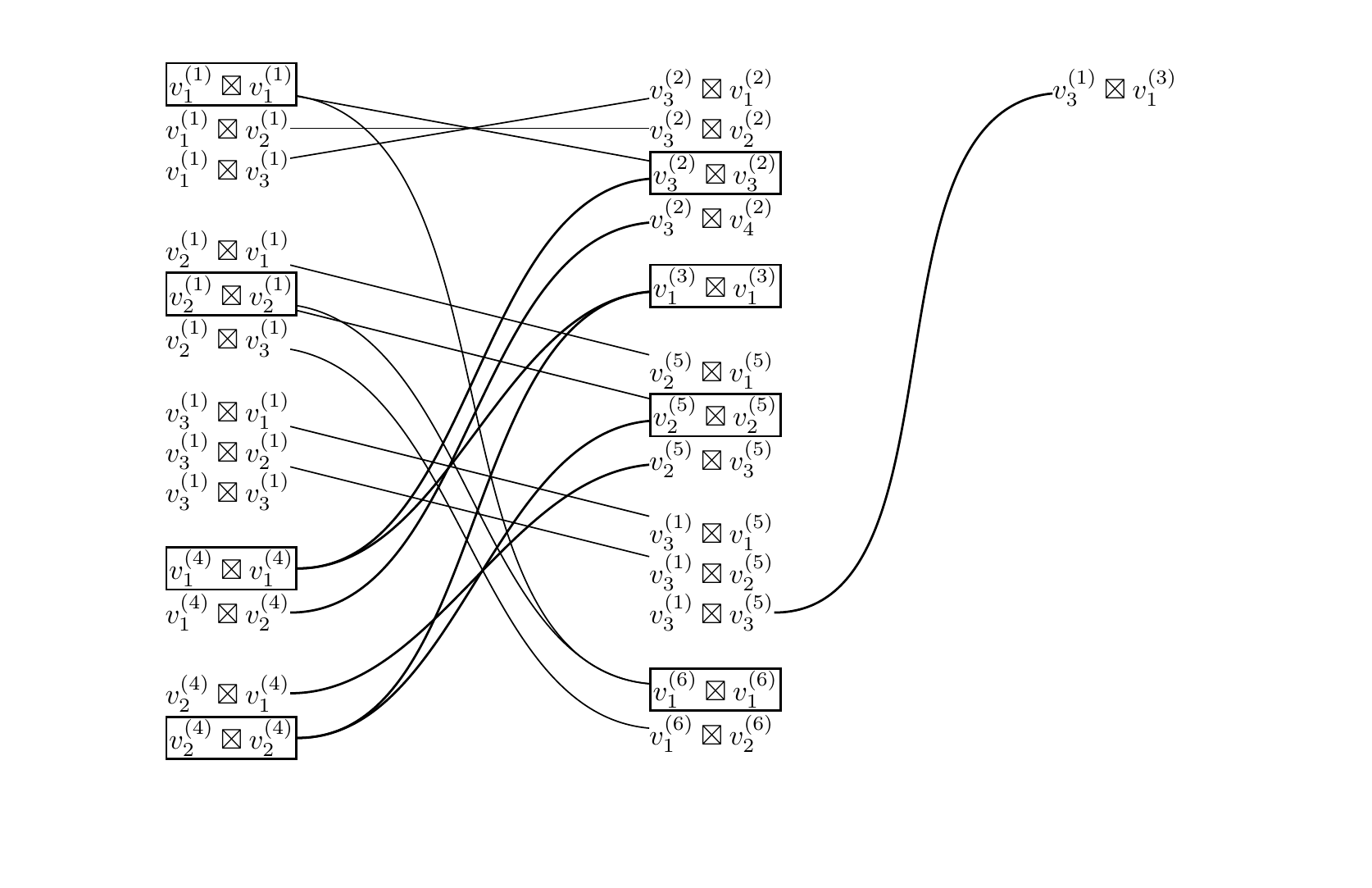}
\end{align*}

\end{document}